\documentclass[12pt]{article}
\usepackage{amsthm,amsfonts,amssymb,epsfig,graphics,amsmath,amsbsy,color,enumerate}

\renewcommand\d{\partial}

\usepackage{subcaption}

\usepackage{calligra}
\theoremstyle{plain}
\newtheorem{theorem}{Theorem}
\newtheorem{lemma}{Lemma}
\newtheorem{proposition}{Proposition}
\newtheorem{remark}{Remark}
\newtheorem{definition}{Definition}
\newtheorem{corollary}{Corollary}

\newcommand{\BbbR}{{\mathbb R}}
\newcommand{\BbbC}{{\mathbb C}}
\newcommand{\mas}{\operatorname{Mas}}

\newcommand{\mor}{\operatorname{Mor}}
\newcommand{\dom}{\operatorname{dom}}
\newcommand{\ran}{\operatorname{ran}}
\newcommand{\colspan}{\operatorname{colspan}}

\DeclareMathOperator{\codim} {codim}

\newcommand\ba{\begin{equation}\begin{aligned}}
\newcommand\ea{\end{aligned}\end{equation}}
\newcommand{\diag}{{\rm diag }}

\newcommand{\bpr}{\begin{proposition}}
\newcommand{\epr}{\end{proposition}}
\newcommand{\bt}{\begin{theorem}}
\newcommand{\et}{\end{theorem}}
\newcommand{\bc}{\begin{corollary}}
\newcommand{\ec}{\end{corollary}}
\newcommand{\bl}{\begin{lemma}}
\newcommand{\el}{\end{lemma}}
\newcommand\br{\begin{remark}}
\newcommand\er{\end{remark}}
\newcommand\bp{\begin{pmatrix}}
\newcommand\ep{\end{pmatrix}}
\newcommand{\be}{\begin{equation}}
\newcommand{\ee}{\end{equation}}
\DeclareMathOperator{\sgn}{sgn}
\newcommand{\sign}{{\text{\rm sign }}}
\newcommand{\Id}{{\rm Id }}

\numberwithin{equation}{section}
\numberwithin{lemma}{section}
\numberwithin{theorem}{section}
\numberwithin{remark}{section}
\numberwithin{claim}{section}
\numberwithin{corollary}{section}
\numberwithin{proposition}{section}
\numberwithin{definition}{section}
\numberwithin{condition}{section}

\title{A Sturm--Liouville theorem for quadratic operator pencils}

\author{Alim Sukhtayev,  Kevin Zumbrun}

\begin{document}

\maketitle

\begin{abstract} 
	We establish a Sturm--Liouville theorem for quadratic operator pencils counting 
	their unstable real roots, with applications to stability of waves.
	Such pencils arise, for example, in reduction of eigenvalue systems to higher-order scalar problems.
\end{abstract}

\section{Introduction}\label{introduction}
In this paper, motivated by recent results of \cite{SYZ18} in a special case,
we establish a general Sturm--Liouville problem for quadratic operator pencils 
on the half- or whole-line.
Specifically, we consider eigenvalue problems on the half line,
\begin{align} \label{main_s}
\begin{split}
&y'' + V(x) y = \lambda f_1(x) y+\lambda^2 f_2(x) y; 
\quad x\in\BbbR_-,\\
&(c+\phi(\lambda))y(0)- y'(0)=0.
\end{split}
\end{align}
and on the whole line,
\begin{equation} \label{main}
y'' + V(x) y = \lambda f_1(x) y+\lambda^2 f_2(x) y; 
\quad x\in\BbbR,
\end{equation}

where $\phi$ 
is a complex analytic matrix-valued function. 
The matrix $c\in M_n(\mathbb{C})$ is Hermitian and $\phi(\lambda)$ is Hermitian for $\lambda\in\mathbb{R}$, $V,f_j \in C(\mathbb{R}, \mathbb{C}^{n\times n})$ are Hermitian potentials. We also list  the following assumptions:

\medskip

\noindent
{\bf (A1)} The limits $\lim_{x \to - \infty} V(x) = V_{-}$ and $\lim_{x \to - \infty} f_j(x) = f_{j-}$ exist,  and $V-V_-, f_j-f_{j-}\in L^1(\BbbR)$, $f_{1}>0$ and $f_{2}\geq\delta>0$, and there is $\gamma<0$ such that, for all $\mu\in\BbbR$ and all $\lambda\in\BbbC$,
\begin{align}
\det(-\mu^2 + V_- -\lambda f_{1-} -\lambda^2 f_{2-})=0
\end{align}
implies 
\begin{equation}
\Re\lambda\leq \gamma<0.
\end{equation}

\noindent
{\bf (A2)} $\phi(0)=0$ and $\phi'(\lambda)<0$ for $\lambda\in\mathbb{R}_+$.

\noindent
{\bf (A3)} $\sign\Im\lambda \Im\phi(\lambda)\leq 0$ for all $\lambda$ with $\Re\lambda\geq0$.\footnote{
	Here and elsewhere $\Im M$ for an operator $M$ is defined as its skew-symmmetric part $\frac12(M^*-M)$.
	Note, for $\Im \lambda=0$, that $\phi(\lambda)$ since Hermitian, automatically satisfied $\Im \phi(\lambda)=0$.}

\noindent
{\bf (A4)} The limits $\lim_{x \to \pm \infty} V(x) = V_{\pm}$ and $\lim_{x \to \pm \infty} f_j(x) = f_{j\pm}$ exist,  and $V-V_\pm, f_j-f_{j\pm}\in L^1(\BbbR)$, $f_{1}>0$ and $f_{2}\geq\delta>0$, and there is $\gamma<0$ such that, for all $\mu\in\BbbR$ and all $\lambda\in\BbbC$,
\begin{align}
\det(-\mu^2 + V_\pm -\lambda f_{1\pm} -\lambda^2 f_{2\pm})=0
\end{align}
implies 
\begin{equation}
\Re\lambda\leq \gamma<0.
\end{equation}

{\bf (A4)}  is related to problem \eqref{main};  {\bf (A1)}, {\bf (A2)} and {\bf (A3)} are related to problem \eqref{main_s}.
Our particular interest lies in counting the number of real nonnegative
eigenvalues of \eqref{main} and \eqref{main_s}. 
As described further in Section \ref{s:disc}, quadratic eigenvalue problems
\eqref{main}--\eqref{main_s} arise for example through reduction of a 
standard eigenvalue system to a higher-order system in a lower-dimensional variable.
As such, their stability has bearing on stability of traveling waves, calculus of variations, etc.
In particular, reduction of a first-order $2\times 2$ system to a second-order scalar problem can always be
performed \cite{JNRYZ18,SYZ18,Y19}, in which case the assumptions of Hermitian coefficients, since they
are real scalar, is automatically satisfied.

We consider at the same time the truncated eigenvalue problems
\begin{align} \label{mainL}
\begin{split}
	&y'' + V(x) y = \lambda f_1(x) y+\lambda^2 f_2(x) y; 
	\quad x\in\BbbR_L:=(-\infty, L],\,\,L\in\mathbb{R},\\
	&y(L)=0.
	\end{split}
\end{align}

Next, we introduce the corresponding operator pencils:

\begin{align}\label{L-}
\begin{split}
&\mathcal{L_-}(\lambda):\dom(\mathcal{L}_-(\lambda))\subset( L^2(\BbbR_-))^n\to( L^2(\BbbR_-))^n,\\
&\mathcal{L_-}(\lambda)y:=y'' + V(x) y- \lambda f_1(x) y-\lambda^2 f_2(x) y, \,y\in\dom(\mathcal{L}_-(\lambda)),
\quad x\in\BbbR_-,\\
&\dom(\mathcal{L}_-(\lambda))=\{y\in( H^2(\BbbR_-))^n:(c+\phi(\lambda))y(0)- y'(0)=0\}.
\end{split}
\end{align}
And
\begin{align}\label{LL}
\begin{split}
&\mathcal{L}_L(\lambda):\dom(\mathcal{L}_L(\lambda))\subset( L^2(\BbbR_L))^n\to( L^2(\BbbR_L))^n,\\
&\mathcal{L}_L(\lambda)y:=y'' + V(x) y- \lambda f_1(x) y-\lambda^2 f_2(x) y, \,y\in\dom(\mathcal{L}_L(\lambda)),
\quad x\in\BbbR_L,\\
&\dom(\mathcal{L}_L(\lambda))=\{y\in( H^2(\BbbR_L))^n: y(L)=0\}.
\end{split}
\end{align}
Finally,
\begin{align} 
\begin{split}
&\mathcal{L}(\lambda):\dom(\mathcal{L}(\lambda))\subset( L^2(\BbbR))^n\to( L^2(\BbbR))^n,\\
&\mathcal{L}(\lambda)y:=y'' + V(x) y- \lambda f_1(x) y-\lambda^2 f_2(x) y, \,y\in\dom(\mathcal{L}(\lambda)),
\quad x\in\BbbR,\\
&\dom(\mathcal{L}(\lambda))=( H^2(\BbbR))^n.
\end{split}
\end{align}

{\it Essential spectrum.} Our 
first
goal is to show that Assumption {\bf (A1)} implies that there exists an open subset $\Omega$ containing the closed right half plane that consists of either points of the resolvent set  or isolated eigenvalues of finite algebraic multiplicity of the operator pencil $\mathcal{L_-}(\cdot)$.

We introduce the closed densely defined operator pencil $\mathcal{T}(\lambda):\mathcal{D(\lambda)}\to\mathcal{H}$, where $\mathcal{D(\lambda)}\subset\mathcal{H}$ is the domain of $\mathcal{T}(\lambda)$.

\begin{definition}(essential spectrum)
	The essential spectrum of $\mathcal{T}$, denoted $\sigma_{ess}(\mathcal{T})$, is the set of all complex numbers $\lambda$ such that $\mathcal{T}(\lambda)$ is not a Fredholm operator with index $0$.
\end{definition}

 Since for the half-line case, the domain of the operator pencil is $\lambda$ dependent, we couldn't find a precise reference for the following lemma which we prove in Appendix \ref{essspect}. 

\begin{lemma}\label{essspec}
	Let Assumption {\bf (A1)} hold. Then $\Omega\subset\BbbC\setminus\sigma_{ess}(\mathcal{L_{-}})$. Moreover, $\Omega$ consists of either points of the resolvent set  or isolated eigenvalues of finite algebraic multiplicity of the operator pencil $\mathcal{L_{-}}(\cdot)$.
\end{lemma}

Similarly,
\begin{lemma}\label{essspecL}
	Let Assumption {\bf (A4)} hold. Then $\Omega\subset\BbbC\setminus\sigma_{ess}(\mathcal{L}_L)$. Moreover, $\Omega$ consists of either points of the resolvent set  or isolated eigenvalues of finite algebraic multiplicity of the operator pencil $\mathcal{L}_L(\cdot)$.
\end{lemma}
 
 For purpose of self-containment, we also provide the proof of the following lemma in Appendix \ref{essspect}. 

\begin{lemma}
	Let Assumption {\bf (A4)} hold. Then $\Omega\subset\BbbC\setminus\sigma_{ess}(\mathcal{L_{}})$. Moreover, $\Omega$ consists of either points of the resolvent set  or isolated eigenvalues of finite algebraic multiplicity of the operator pencil $\mathcal{L_{}}(\cdot)$.
\end{lemma}
{\it Maslov index.} As a starting point, we define what we will mean by a {\it Lagrangian
subspace} of $\mathbb{C}^{2n}$.

\begin{definition} \label{lagrangian_subspace}
We say $\ell \subset \mathbb{C}^{2n}$ is a Lagrangian subspace of $\mathbb{C}^{2n}$
if $\ell$ has dimension $n$ and
\begin{equation} 
(J_{2n} u, v)_{\mathbb{C}^{2n}} = 0, 
\end{equation} 
for all $u, v \in \ell$. Here, $(\cdot, \cdot)_{\mathbb{C}^{2n}}$ denotes
the standard inner product on $\mathbb{C}^{2n}$. In addition, we denote by 
$\Lambda (n)$ the collection of all Lagrangian subspaces of $\mathbb{C}^{2n}$, 
and we will refer to this as the {\it Lagrangian Grassmannian}. 
\end{definition}

Any Lagrangian subspace of $\mathbb{C}^{2n}$ can be
spanned by a choice of $n$ linearly independent vectors in 
$\mathbb{C}^{2n}$. We will generally find it convenient to collect
these $n$ vectors as the columns of a $2n \times n$ matrix $\mathbf{X}$, 
which we will refer to as a {\it frame} for $\ell$. Moreover, we will 
often coordinatize our frames as $\mathbf{X} = {X \choose Y}$, where $X$ and $Y$ are 
$n \times n$ matrices.

Suppose $\ell_1 (\cdot), \ell_2 (\cdot)$ denote paths of Lagrangian 
subspaces $\ell_i: \mathcal{I} \to \Lambda (n)$, for some parameter interval 
$\mathcal{I}$. The Maslov index associated with these paths, which we will 
denote $\mas (\ell_1, \ell_2; \mathcal{I})$, is a count of the number of times
the paths $\ell_1 (\cdot)$ and $\ell_2 (\cdot)$ intersect, counted
with both multiplicity and direction. (In this setting, if we let 
$t_*$ denote the point of intersection (often referred to as a 
{\it conjugate point}), then multiplicity corresponds with the dimension 
of the intersection $\ell_1 (t_*) \cap \ell_2 (t_*)$; a precise definition of what we 
mean in this context by {\it direction} will be
given in Section \ref{maslov_section}.) In some cases, the Lagrangian 
subspaces will be defined along some path in the 
$(\alpha,\beta)$-plane 
\begin{equation*}
\Gamma = \{ (\alpha (t), \beta (t)): t \in \mathcal{I}\},
\end{equation*}    
and when it is convenient we will use the notation 
$\mas (\ell_1, \ell_2; \Gamma)$. 

 We say that the evolution of $\mathcal{L} = (\ell_1, \ell_2)$
is {\it monotonic} provided all intersections occur with 
the same direction. If the intersections all correspond with 
the positive direction, then we can compute
\begin{equation*}
\mas (\ell_1, \ell_2; \mathcal{I}) 
= \sum_{t \in \mathcal{I}} \dim (\ell_1 (t) \cap \ell_2(t)).
\end{equation*}
Suppose $\mathbf{X}_1 (t) = {X_1 (t) \choose Y_1 (t)}$ and 
$\mathbf{X}_2 (t) = {X_2 (t) \choose Y_2 (t)}$ respectively 
denote frames for Lagrangian subspaces of $\mathbb{C}^{2n}$,
$\ell_1 (t)$ and $\ell_2(t)$. Then we can express
this last relation as 
\begin{equation*}
\mas (\ell_1, \ell_2; \mathcal{I}) 
= \sum_{t \in \mathcal{I}} \dim \ker (\mathbf{X}_1 (t)^* J \mathbf{X}_2 (t)).
\end{equation*}

\subsection{Main results}\label{s:main}
We establish the following generalized Sturm-Liouville theorems relating the spectral count, or number of
real eigenvalues greater than a given {\it nonnegative}
value $\lambda$, to the number of conjugate points of the Lagrangian frame
asymptotic to the decaying eigenspace at $x\to -\infty$ plus, in the case of the half-line problem \eqref{main_s},
a computable boundary correction term.
Note that, in contrast to the standard case of a linear operator pencil, 
we do not obtain information for negative $\lambda$, but only about the number of possible unstable real eigenvalues
$\lambda>0$.
Nonetheless, this is sufficient to determine stability or instability of spectra, 
which is typically the question of interest.

\begin{theorem} \label{bc2_theorem} 
	For equation (\ref{main_s}), let Assumptions {\bf (A1)}, {\bf (A2)} hold, 
	and let $\mathbf{X}$ and $\mathbf{X}_\phi$ denote
	the Lagrangian frames corresponding the unstable subpsace $E_-^u(x,\lambda)$ and the $\Phi(\lambda)$ subspace $\colspan{I_n \choose c+ \phi(\lambda)}$, respectively.  If 
	$\mathcal{N} (\lambda)$ denotes the spectral
	count for (\ref{main_s}) (the number of real eigenvalues that are greater than $\lambda$),
	then 
	\begin{align*}
	\mathcal{N} (\lambda)
	&=-\mas (E_-^u(\cdot; \lambda), \Phi(\lambda); [-\infty,0])-\dim(\ker(\mathcal{L}_-(\lambda)))\\&+\mor(\sqrt{\lambda f_{1-}+\lambda^2 f_{2-}-V_-}-c-\phi(\lambda))+\dim\ker(\sqrt{\lambda f_{1-}+\lambda^2 f_{2-}-V_-}-c-\phi(\lambda)),\\
	\,\,\lambda\geq0&.
	\end{align*}
\end{theorem}

\begin{theorem} \label{bc3_theorem} 
	For equation (\ref{mainL}), let Assumption {\bf (A4)} hold, 
	and let $\mathbf{X}$ and $\mathbf{X}_D$ denote
	the Lagrangian frames corresponding to the unstable subpsace $E_-^u(x,\lambda)$ and the Dirichlet subspace $\mathcal{D}$, respectively.  If 
	$\mathcal{N} (\lambda)$ denotes the spectral
	count for (\ref{main}) (the number of real eigenvalues that are greater than $\lambda$),
	then 
	\begin{equation*}
	\mathcal{N} (\lambda)
	=\sum_{-\infty < x <L} \dim (E_-^u(x,\lambda) \cap \mathcal{D})= \sum_{x \in (-\infty,L)} \dim \ker (\mathbf{X} (x; \lambda)^* J \mathbf{X}_D),\,\,\lambda\geq0.
	\end{equation*}
\end{theorem}

\begin{theorem} \label{bc1_theorem} 
For equation (\ref{main}), let Assumption {\bf (A4)} hold, 
and let $\mathbf{X}$ and $\mathbf{X}_D$ denote
the Lagrangian frames corresponding to the unstable subpsace $E_-^u(x,\lambda)$ and the Dirichlet subspace $\mathcal{D}$, respectively.  If 
$\mathcal{N} (\lambda)$ denotes the spectral
count for (\ref{main}) (the number of real eigenvalues that are greater than $\lambda$),
then 
\begin{equation*}
\mathcal{N} (\lambda)
=\sum_{x\in\BbbR} \dim (E_-^u(x,\lambda) \cap \mathcal{D})= \sum_{x \in \BbbR} \dim \ker (\mathbf{X} (x; \lambda)^* J \mathbf{X}_D),\,\,\lambda\geq0.
\end{equation*}
\end{theorem}

{\bf Typical examples of the eigenvalue curves} \\

\begin{figure}
	\centering
	\begin{subfigure}{0.5\textwidth}
		\centering
		\includegraphics[width=.8\linewidth]{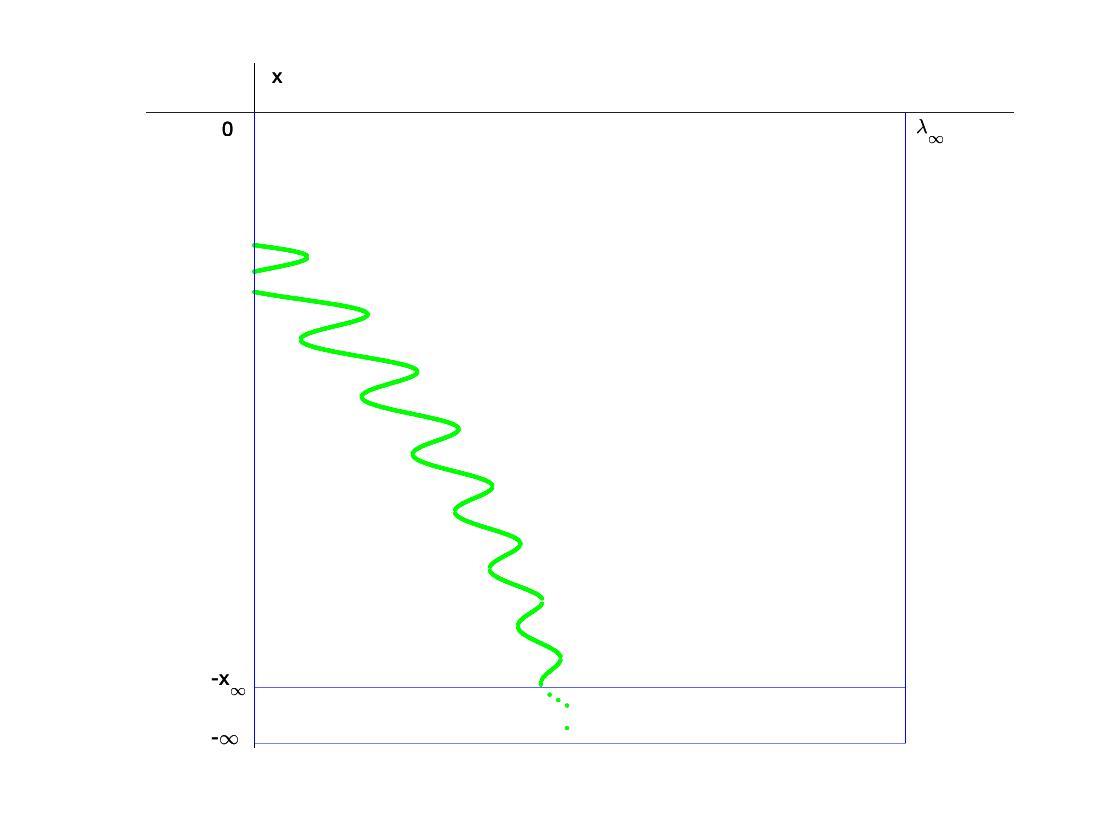}
		\label{fig:sub1}
	\end{subfigure}%
	\begin{subfigure}{.5\textwidth}
		\centering
		\includegraphics[width=0.95\linewidth]{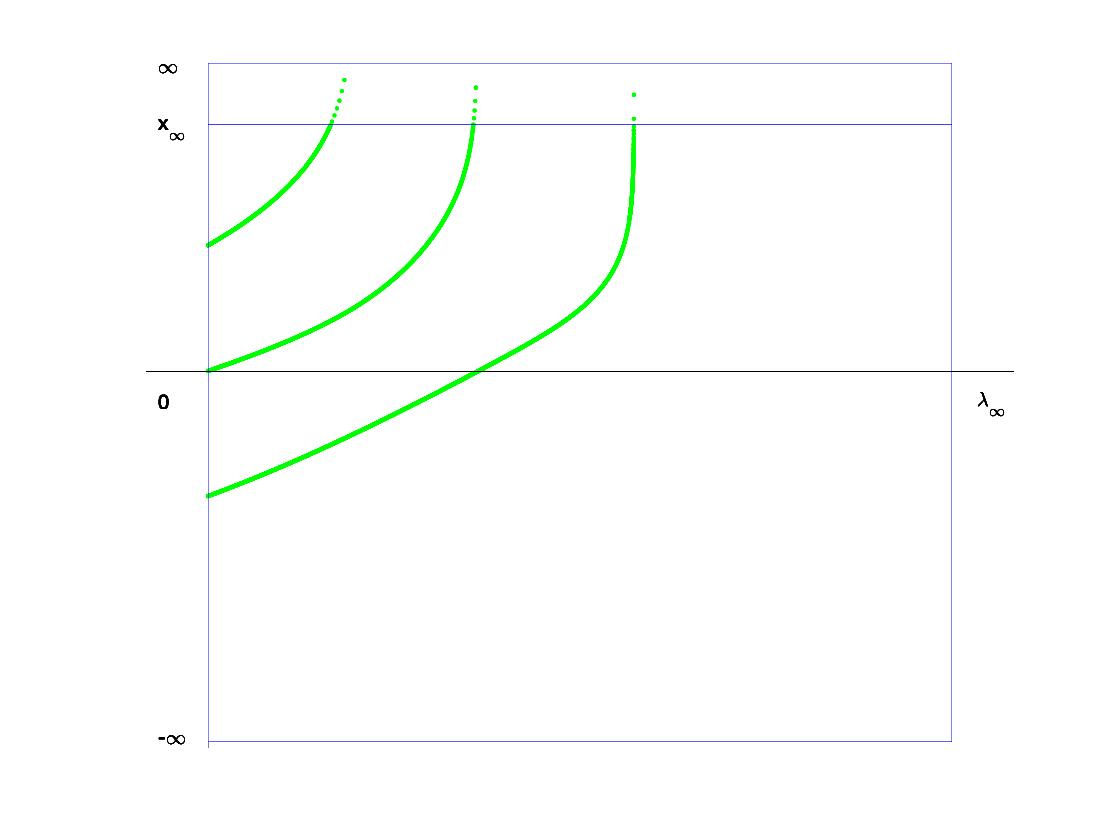}
	\end{subfigure}
	\caption{Eigenvalue curves }\label{scalar}
\end{figure}

{\bf Example 1} (half-line, scalar)
We consider the potentials\\ $V(x)=-1-(815+219\cos(1.8x))e^{0.1x} $, $f_1=1$, $f_2=2$ along with the boundary condition $(18-9\lambda)y(0)-y'(0)=0$. In this case, we see
the emergence of an eigenvalue from the bottom shelf, and we notice a very distinct loss of the monotonicity. See the left-half of Figure \ref{scalar}. The Maslov Index in this case is
$1$, the Morse index of $\sqrt{-V_-}-c$ is $1$, and according to \ref{bc2_theorem}, 
this means that $\mathcal{N} (0)=0$  (the number of real eigenvalues for the problem \eqref{main_s} that are greater than $0$).

{\bf Example 2} (full-line, scalar)
We consider the potentials $V(x)=-1+1.8e^{-0.06|x|}$, $f_1=1$, $f_2=2$. In this case, there
can be no crossings along the bottom shelf, and indeed the only allowable behavior is for the
eigenvalue curves to enter the box through the curve $\lambda=0$ and move upward until reaching the curve $x=\infty$.
See the right-half of Figure \ref{scalar}.  We ran our numerics up to some big positive value $x_{\infty}$. The number of intersections of the unstable subpsace $E_-^u(x,0)$ and the Dirichlet subspace $\mathcal{D}$ is 3, and according to Theorem \ref{bc1_theorem}
this means that $\mathcal{N} (0)=3$  (the number of real eigenvalues for the problem \eqref{main} that are greater than $0$).

{\bf Example 3} (half-line, $2\times2$ system)
We consider the potentials\\ $V(x)=\begin{pmatrix}
-1-(815+219\cos(1.8x))e^{0.1x} & 0 \\
0 & -1-(255+0.1\cos(0.5x))e^{0.15x}
\end{pmatrix}$, $f_1=I$ and $f_2=2I$ along with the boundary matrices $\phi(\lambda)=\begin{pmatrix}
-9\lambda & 0 \\
0 & -9\lambda
\end{pmatrix}$ and $c=\begin{pmatrix}
18 & 2 \\
2 & 25
\end{pmatrix}$. Note that the coupling appears via the matrix $c$. See the left-half of Figure \ref{matrix}. The Maslov Index in this case is
$1$, the Morse index of $\sqrt{-V_-}-c$ is $2$, and according to \ref{bc2_theorem}, 
this means that $\mathcal{N} (0)=1$  (the number of real eigenvalues for the problem \eqref{main_s} that are greater than $0$).

{\bf Example 4} (full-line, $2\times2$ system)
We consider the potentials\\ $V(x)=\begin{pmatrix}
-1+1.93e^{-0.141|x|} & 0.5 \\
0.5 & -1+1.93e^{-0.141|x|}
\end{pmatrix}$, $f_1=I$ and $f_2=2I$. Note that the coupling appears via the potential $V$. See the right-half of Figure \ref{matrix}. The number of intersections of the unstable subpsace $E_-^u(x,0)$ and the Dirichlet subspace $\mathcal{D}$ is 5, and according to Theorem \ref{bc1_theorem}
this means that $\mathcal{N} (0)=5$  (the number of real eigenvalues for the problem \eqref{main} that are greater than $0$). Also, note that for the systems, the eigenvalue curves might intersect which can be observed for our particular $2\times2$ system.
\begin{figure}
	\centering
	\begin{subfigure}{0.5\textwidth}
		\centering
		\includegraphics[width=.8\linewidth]{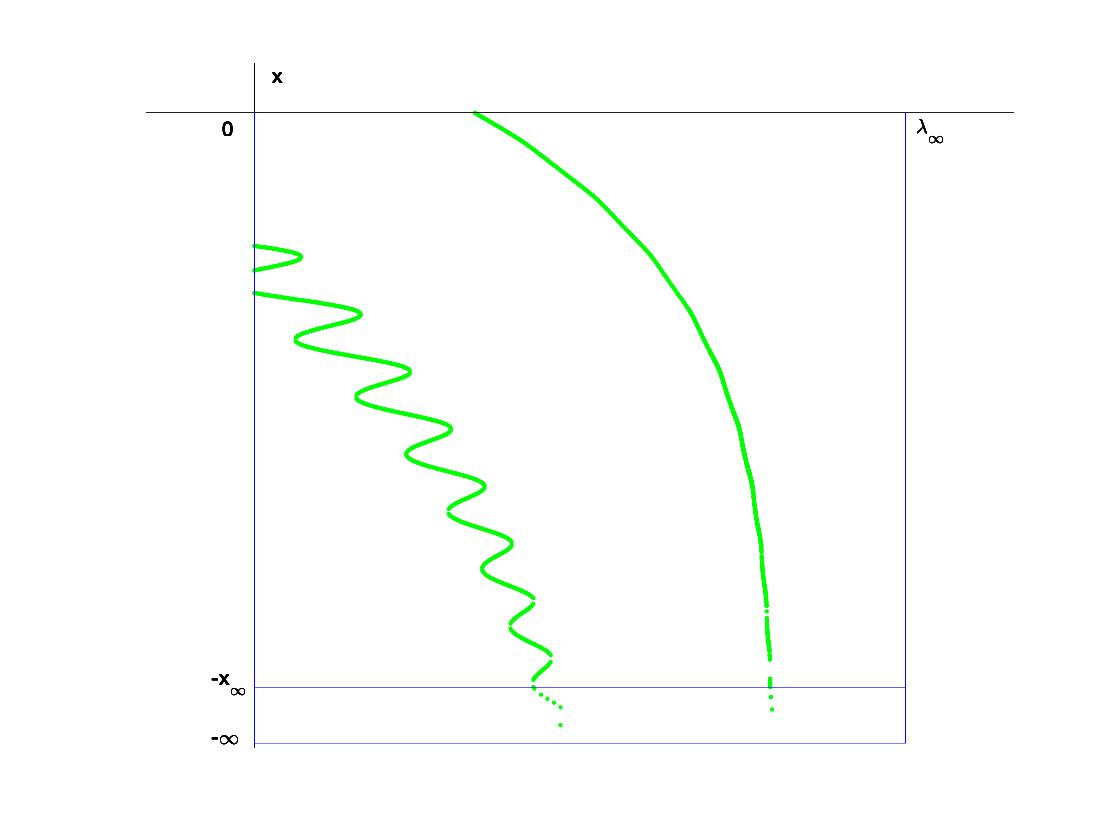}
	\end{subfigure}%
	\begin{subfigure}{.5\textwidth}
		\centering
		\includegraphics[width=0.95\linewidth]{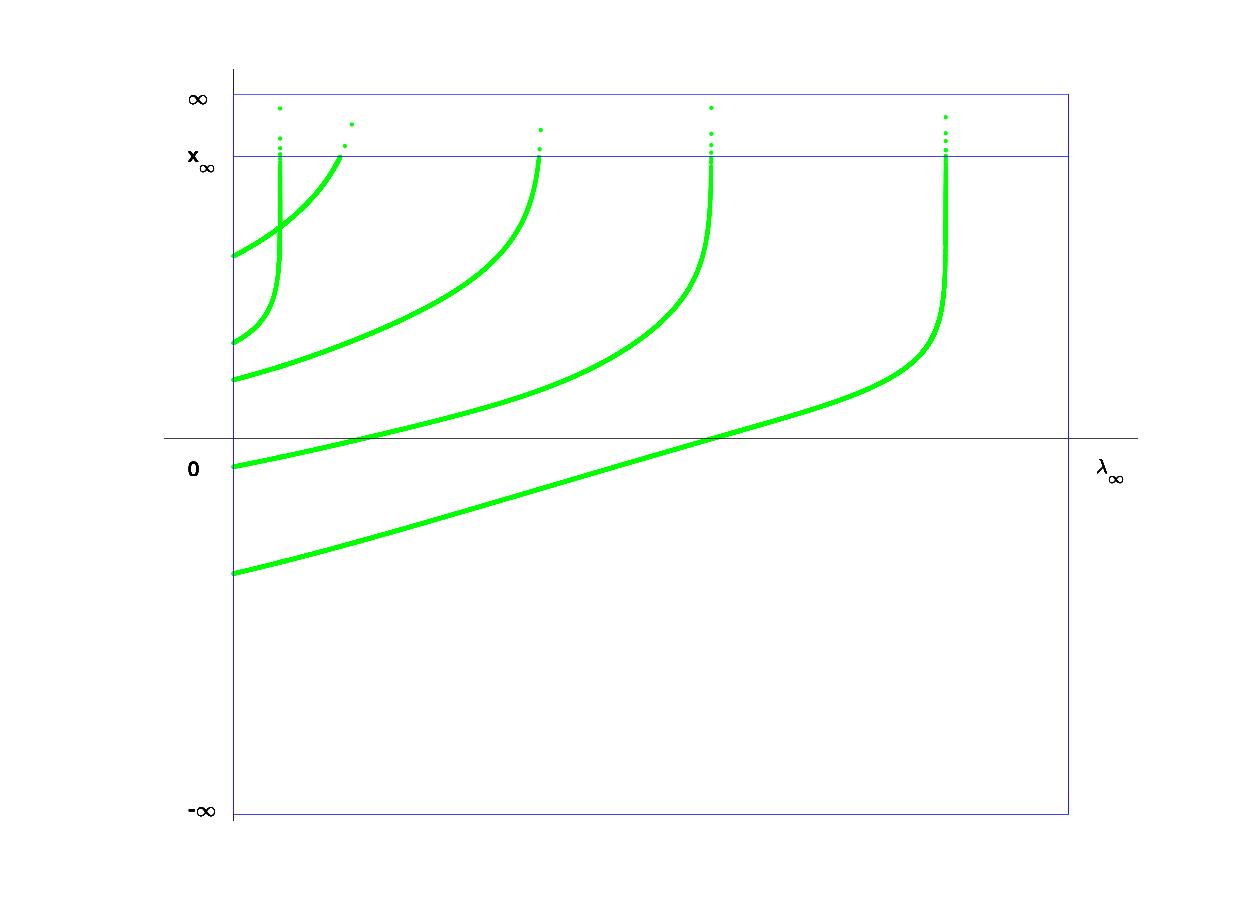}
	\end{subfigure}
	\caption{Eigenvalue curves }\label{matrix}
\end{figure}

\subsection{Reality of eigenvalues}\label{s:reality}
The above theorems concern only the real spectrum of the associated operator pencil.
However, adapting an argument of \cite[Lemma 4.1]{SYZ18} similar to that for the classic linear pencil case,
we may readily see that nonstable spectra $\Re \lambda \geq 0$ of the whole-line problem are necessarily real, hence
our conclusions are decisive for stability.  Likewise, for the half-line problem, unstable spectra are real under
the additional (sharp, see Remark \ref{phirmk}) assumption that $\Im \phi(\lambda)\Im \lambda < 0$ for $\Re \lambda \geq 0$;
	see Lemma \ref{nocomplex}.

\section{The Maslov Index on $\mathbb{C}^{2n}$} \label{maslov_section}

Suppose $\mathbf{X}= {X \choose Y}$,
$\mathbf{X}_D = {0 \choose I_n}$ and $\mathbf{X}_\phi = {I_n \choose c+ \phi(\lambda)}$ respectively denote frames for 
Lagrangian subspaces  $\ell(x,\lambda):=E_-^u(x,\lambda)$, the Dirichlet subspace $\mathcal{D}$ and the $\Phi(\lambda)$ subspace. We now set 
\begin{align}\label{WD}
\begin{split}
\tilde{W}_D :&= (X + iY)(X-iY)^{-1},\\
\tilde{W}_\phi :&= -(X + iY)(X-iY)^{-1}(X_\phi - iY_\phi)(X_\phi+iY_\phi)^{-1},\\
\end{split}
\end{align}
noting that $\tilde{W}_D$ and $\tilde{W}_\phi$ detect intersections of $\ell = \colspan(\mathbf{X})=E_-^u(x,\lambda)$ 
with the Dirichlet subspace and the $\phi$ subspace, respectively. Moreover,

\begin{align*}
\begin{split}
\dim \ker (\mathbf{X}^* J \mathbf{X}_D) &= \dim \ker (\tilde{W}_D + I),\\
\dim \ker (\mathbf{X}^* J \mathbf{X}_\phi) &= \dim \ker (\tilde{W}_\phi + I),\\
\end{split}
\end{align*}

In general, given any two Lagrangian subspaces $\ell_1$ and $\ell_2$, with associated 
frames $\mathbf{X}_1 = {X_1 \choose Y_1}$ and 
$\mathbf{X}_2 = {X_2 \choose Y_2}$, we can define the complex $n \times n$
matrix 
\begin{equation} \label{tildeW}
\tilde{W} = - (X_1 + i Y_1) (X_1 - i Y_1)^{-1} (X_2 - i Y_2) (X_2 + i Y_2)^{-1}.
\end{equation}

Given two continuous maps $\ell_1 (t), \ell_2(t)$ on a parameter
interval $\mathcal{I}$, we denote by $\mathcal{L}(t)$ the path 
\begin{equation*}
\mathcal{L} (t) = (\ell_1 (t), \ell_2(t)).
\end{equation*} 
In what follows, we will define the Maslov index for the path 
$\mathcal{L} (t)$, which will be a count, including both multiplicity
and direction, of the number of times the Lagrangian paths
$\ell_1$ and $\ell_2$ intersect. In order to be clear about 
what we mean by multiplicity and direction, we observe that 
associated with any path $\mathcal{L} (t)$ we will have 
a path of unitary complex matrices as described in (\ref{tildeW}).
We have already noted that the Lagrangian subspaces $\ell_1$
and $\ell_2$ intersect at a value $t_* \in \mathcal{I}$ if and only 
if $\tilde{W} (t_*)$ has -1 as an eigenvalue. (Recall that we refer to the
value $t_*$ as a {\it conjugate point}.) In the event of 
such an intersection, we define the multiplicity of the 
intersection to be the multiplicity of -1 as an eigenvalue of 
$\tilde{W}$ (since $\tilde{W}$ is unitary the algebraic and geometric
multiplicites are the same). When we talk about the direction 
of an intersection, we mean the direction the eigenvalues of 
$\tilde{W}$ are moving (as $t$ varies) along the unit circle 
$S^1$ when they cross $-1$ (we take counterclockwise as the 
positive direction). We note that we will need to take care with 
what we mean by a crossing in the following sense: we must decide
whether to increment the Maslov index upon arrival or 
upon departure. Indeed, there are several different approaches 
to defining the Maslov index (see, for example, \cite{CLM, rs93}), 
and they often disagree on this convention. 

Following \cite{BF98, F, P96} (and in particular Definition 1.5 
from \cite{BF98}), we proceed by choosing a 
partition $a = t_0 < t_1 < \dots < t_n=b$ of $\mathcal{I} = [a,b]$, along 
with numbers $\epsilon_j \in (0,\pi)$ so that 
$\ker\big(\tilde{W} (t) - e^{i (\pi \pm \epsilon_j)} I\big)=\{0\}$
for $t_{j-1} \le t \le t_j$; 
that is, $e^{i(\pi \pm \epsilon_j)} \in \mathbb{C} \setminus \sigma(\tilde{W} (t))$, 
for $t_{j-1} \le t \le t_j$ and $j=1,\dots,n$. 
Moreover, we notice that for each $j=1,\dots,n$ and any 
$t \in [t_{j-1},t_j]$ there are only 
finitely many values $\beta \in [0,\epsilon_j)$ 
for which $e^{i(\pi+\beta)} \in \sigma(\tilde{W} (t))$.

Fix some $j \in \{1, 2, \dots, n\}$ and consider the value
\begin{equation} \label{kdefined}
k (t,\epsilon_j) := 
\sum_{0 \leq \beta < \epsilon_j}
\dim \ker \big(\tilde{W} (t) - e^{i(\pi+\beta)}I \big).
\end{equation} 
for $t_{j-1} \leq t \leq t_j$. This is precisely the sum, along with multiplicity,
of the number of eigenvalues of $\tilde{W} (t)$ that lie on the arc 
\begin{equation*}
A_j := \{e^{i t}: t \in [\pi, \pi+\epsilon_j)\}.
\end{equation*}
The stipulation that 
$e^{i(\pi\pm\epsilon_j)} \in \mathbb{C} \setminus \sigma(\tilde{W} (t))$, for 
$t_{j-1} \le t \le t_j$
ensures that no eigenvalue can enter $A_j$ in the clockwise direction 
or exit in the counterclockwise direction during the interval $t_{j-1} \le t \le t_j$.
In this way, we see that $k(t_j, \epsilon_j) - k (t_{j-1}, \epsilon_j)$ is a 
count of the number of eigenvalues that enter $A_j$ in the counterclockwise 
direction (i.e., through $-1$) minus the number that leave in the clockwise direction
(again, through $-1$) during the interval $[t_{j-1}, t_j]$.


In dealing with the catenation of paths, it's particularly important to 
understand the difference $k(t_j, \epsilon_j) - k (t_{j-1}, \epsilon_j)$ if an 
eigenvalue resides at $-1$ at either $t = t_{j-1}$
or $t = t_j$ (i.e., if an eigenvalue begins or ends at a crossing). If an eigenvalue 
moving in the counterclockwise direction 
arrives at $-1$ at $t = t_j$, then we increment the difference forward, while if 
the eigenvalue arrives at -1 from the clockwise direction we do not (because it
was already in $A_j$ prior to arrival). On
the other hand, suppose an eigenvalue resides at -1 at $t = t_{j-1}$ and moves
in the counterclockwise direction. The eigenvalue remains in $A_j$, and so we do not increment
the difference. However, if the eigenvalue leaves in the clockwise direction 
then we decrement the difference. In summary, the difference increments forward upon arrivals 
in the counterclockwise direction, but not upon arrivals in the clockwise direction,
and it decrements upon departures in the clockwise direction, but not upon 
departures in the counterclockwise direction.      

We are now ready to define the Maslov index.

\begin{definition} \label{dfnDef3.6}  
Let $\mathcal{L} (t) = (\ell_1 (t), \ell_2(t))$, where 
$\ell_1, \ell_2:\mathcal{I} \to \Lambda (n)$ 
are continuous paths in the Lagrangian--Grassmannian. 
The Maslov index $\mas(\mathcal{L}; \mathcal{I})$ is defined by
\begin{equation}
\mas(\mathcal{L}; \mathcal{I})=\sum_{j=1}^n(k(t_j,\epsilon_j)-k(t_{j-1},\epsilon_j)).
\end{equation}
\end{definition}

\begin{remark}
As we did in the introduction, we will typically refer explicitly to 
the individual paths with the notation $\mas (\ell_1, \ell_2; \mathcal{I})$.
\end{remark}

\begin{remark} As discussed in \cite{BF98}, the Maslov index does not depend
on the choices of $\{t_j\}_{j=0}^n$ and $\{\epsilon_j\}_{j=1}^n$, so long as 
these choices follow the specifications described above. 
\end{remark}

\subsection{Direction of Rotation} \label{rotation_section}

As noted in the previous section, the direction we associate with a 
conjugate point is determined by the direction in which eigenvalues
of $\tilde{W}$ rotate through $-1$ (counterclockwise is positive, 
while clockwise is negative). When analyzing the Maslov index, we 
need a convenient framework for analyzing this direction, and 
the development of such a framework is the goal of this section. 

\begin{lemma}[\cite{HS18}] \label{monotonicity1}
	Suppose $\ell_1, \ell_2: I \to \Lambda (n)$ denote paths of 
	Lagrangian subspaces of $C^{2n}$ with absolutely continuous frames $\mathbf{X}_1 = {X_1 \choose Y_1}$
	and $\mathbf{X}_2 = {X_2 \choose Y_2}$ (respectively). If there exists 
	$\delta > 0$ so that the matrices 
	\begin{equation*}
	- \mathbf{X}_1^* J \mathbf{X}_1' = X_1 (t)^* Y_1' (t) - Y_1 (t)^* X_1'(t)
	\end{equation*}
	and (noting the sign change)
	\begin{equation*}
	\mathbf{X}_2^* J \mathbf{X}_2' = - (X_2 (t)^* Y_2' (t) - Y_2 (t)^* X_2'(t))
	\end{equation*}
	are both a.e.-non-negative in $(t_0-\delta,t_0+\delta)$, and at least one is 
	a.e.-positive definite then the eigenvalues of $\tilde{W} (t)$ rotate in the 
	counterclockwise direction as $t$ increases through $t_0$. 
	Likewise, if both of these matrices are a.e.-non-positive, and at least one is 
	a.e.-negative definite, 
	then the eigenvalues of $\tilde{W} (t)$ rotate in the clockwise direction as 
	$t$ increases through $t_0$.
\end{lemma}

\section{Proof of Theorem \ref{bc2_theorem} } 

\subsection{Upper Bound on the Spectrum of \eqref{main_s}} \label{bound_section}

By Lemma \ref{essspec}, we know that  the real part of the essential spectrum of \eqref{main_s} is bounded above by $\gamma<0$. Next, we show that a set of the real eigenvalues of \eqref{main_s} is bounded above.

\begin{lemma}\label{nu}
	Assume {\bf (A1)} and {\bf (A2)}. Then there exists $\nu\in\BbbR$ such that for all real eigenvalues $\lambda$ of \eqref{main_s}
	\begin{equation}
	\lambda\le\frac{\sqrt{\|f_2\|_{L^{\infty} (\mathbb{R}_-)}|\nu|}}{\delta}.
	\end{equation}
\end{lemma}

\begin{proof}
	Let $\lambda$ be a real eigenvalue of \eqref{main_s} with the corresponding eigenvector $y$. Then 
	\begin{align} 
	\begin{split}
	&y'' + V(x) y = \lambda f_1(x) y+\lambda^2 f_2(x) y; 
	\quad x\in\BbbR_-,\\
	&(c+\phi(\lambda))y(0)- y'(0)=0.
	\end{split}
	\end{align}
	Thus, after multiplying by $y$ and integration by parts, we arrive at
	\begin{align} 
	\begin{split}
	&\big((c+\phi(\lambda))y(0),y(0)\big)  -(y',y')+ (Vy,y) = \lambda (f_1y,y)+\lambda^2 (f_2y,y),
	\end{split}
	\end{align}
	or, rearranging,
	\begin{align} 
	\begin{split}
	(f_2y,y) \lambda^2 +(f_1y,y) \lambda+(y',y')-(Vy,y)-\big((c+\phi(\lambda))y(0),y(0)\big)=0.
	\end{split}
	\end{align}
	Therefore, $\lambda\in\BbbR$ satisfies one of the following equalities
	\begin{align}\label{disc}
	\begin{split}
	\lambda=\frac{-(f_1y,y)\pm\sqrt{(f_1y,y)^2-4(f_2y,y)[(y',y')-(Vy,y)-\big((c+\phi(\lambda))y(0),y(0)\big)]}}{2(f_2y,y)}.
	\end{split}
	\end{align}
	If $\lambda$ satisfies the equality with the negative sign in front of the square root, then $\lambda$ is nonpositive.
	Thus, we may assume that $\lambda$  satisfies the equality with the positive sign in front of the square root.
	
	Now, we estimate the following quadratic form $\mathcal{M}$ with the domain $H^1(\mathbb{R}_-)$:
	\begin{equation}
	\mathcal{M}[y]=(y',y')-(Vy,y)-\big((c+\phi(\lambda))y(0),y(0)\big),\,\,\,\lambda\in\mathbb{R}.
	\end{equation}
	Since $\phi(\lambda)$ is Hermitian and $\phi'(\lambda)<0$ for $\lambda\in\mathbb{R}_+$, by \cite[Theorem 5.4.]{Kato}, we conclude that $\phi(\lambda)\leq0$ for $\lambda\in\mathbb{R}_+$. Hence,
	\begin{equation}
	\mathcal{M}[y]\geq(y',y')-(Vy,y)-\big(cy(0),y(0)\big),\,\,\,\lambda\in\mathbb{R}.
	\end{equation}
	
	Given any $\epsilon > 0$ 
	there is a corresponding $\beta (\epsilon)>0$ so that 
	\begin{equation*}
	|y(0)|^2 \le
	\epsilon \|y'\|_{L^2 (\mathbb{R}_-)}^2 + \beta (\epsilon) \|y\|_{L^2 (\mathbb{R}_-)}^2.
	\end{equation*}
	
	Choose $\epsilon > 0$ small enough so that $||c|| \epsilon < 1$. Then (see \cite{HLS})
	
	\begin{equation*}
	\begin{aligned}
	{\mathcal{M}}[y] &\ge
	\|y'\|_{L^2 (\mathbb{R}_-)}^2 - \|V\|_{L^{\infty} (\mathbb{R}_-)} \|y\|_{L^2 (\mathbb{R}_-)}^2 -||c|| \epsilon\|y'\|_{L^2 (\mathbb{R}_-)}^2 -||c||\beta (\epsilon) \|y\|_{L^2 (\mathbb{R}_-)}^2 \\
	&= ( 1-||c|| \epsilon) \|y'\|_{L^2 (\mathbb{R}_-)}^2 
	+ \Big(- \|V\|_{L^{\infty} (\mathbb{R}_-)}-||c||\beta (\epsilon)\Big) \|y\|_{L^2 (\mathbb{R}_-)}^2\\
	&\ge\Big(- \|V\|_{L^{\infty} (\mathbb{R}_-)}-||c||\beta (\epsilon)\Big) \|y\|_{L^2 (\mathbb{R}_-)}^2. 
	\end{aligned}
	\end{equation*}
	Therefore, if $\nu=(- \|V\|_{L^{\infty} (\mathbb{R}_-)}-||c||\beta (\epsilon))$ which is independent of $\lambda$ and $y\in H^1(\mathbb{R}_-)$, then  ${\mathcal{M}}\ge\nu$. Thus, we have
	\begin{align}
	\begin{split}
	&-(f_1y,y)+\sqrt{(f_1y,y)^2-4(f_2y,y)[(y',y')-(Vy,y)-\big((c+\phi(\lambda))y(0),y(0)\big)]}\\
	&\le-(f_1y,y)+\sqrt{(f_1y,y)^2-4(f_2y,y)\nu\|y\|_{L^2 (\mathbb{R}_-)}^2]}\le2\sqrt{\|f_2\|_{L^{\infty} (\mathbb{R}_-)}|\nu|}\|y\|_{L^2 (\mathbb{R}_-)},
	\end{split}
	\end{align}
	and therefore, $\lambda\le\frac{\sqrt{\|f_2\|_{L^{\infty} (\mathbb{R}_-)}|\nu|}}{\delta}$.
\end{proof}

\begin{remark}\label{trunc}
	We introduce the truncated eigenvalue problem
	\begin{align}\label{x0}
	\begin{split}
	&y'' + V(x) y = \lambda f_1(x) y+\lambda^2 f_2(x) y; 
	\quad x\in\BbbR_{x_0}:=(-\infty,x_0],\\
	&(c+\phi(\lambda))y(x_0)- y'(x_0)=0.
	\end{split}
	\end{align}
	
	According to the proof of Lemma \ref{nu}, we have the uniform upper bound  estimate (independent of $x_0$) for any real eigenvalue of \eqref{x0}, that is,
	\begin{equation*}
	\lambda\le\frac{\sqrt{\|f_2\|_{L^{\infty} (\mathbb{R}_-)}|\nu|}}{\delta}.
	\end{equation*}
	We also have the upper bound $\frac{\sqrt{\|f_{2-}\|_{L^{\infty} (\mathbb{R}_-)}|\nu_-|}}{\delta}$ ($\nu_-:=- \|V_-\|_{L^{\infty} (\mathbb{R}_-)}-||c||\beta (\epsilon)$) for any real eigenvalue of the constant-coefficient problem
	\begin{align}\label{concoeff}
	\begin{split}
	&y'' + V_- y = \lambda f_{1-} y+\lambda^2 f_{2-} y; 
	\quad x\in\BbbR_{-},\\
	&(c+\phi(\lambda))y(0)- y'(0)=0.
	\end{split}
	\end{align}
\end{remark}

\subsection{Positivity of the derivative of the matrix square root}

\begin{lemma}\label{positivity}
	Let $M\in C^1(\mathbb{R}, \mathbb{C}^{n\times n})$, and assume that $M(\lambda)$ and $M'(\lambda)$ are Hermitian and positive definite for $\lambda\in\mathbb{R}$. Then $(\sqrt{M(\lambda)})'>0$ for $\lambda\in\mathbb{R}$.
\end{lemma}

\begin{proof}
	We have
	\begin{equation}
	(M^{1/2})'M^{1/2}+M^{1/2}(M^{1/2})'=M'
	\end{equation}
	Multiply both sides by $M^{-1/4}$ from the right and the left
	\begin{equation}
	M^{-1/4}(M^{1/2})'M^{1/4}+M^{1/4}(M^{1/2})'M^{-1/4}=M^{-1/4}M'M^{-1/4}
	\end{equation}
	Let $C:=M^{-1/4}(M^{1/2})'M^{1/4}$. Since $M^{-1/4}M'M^{-1/4}>0$, we have $C+C^*>0$. Therefore, we have the estimate on the real part of the spectrum of $C$, that is, $\Re(\sigma(C))>0$. We also know that $C$ is similar to $(M^{1/2})'$. Hence, $\Re(\sigma((M^{1/2})'))>0$, or $\sigma((M^{1/2})')>0$ ($(M^{1/2})'$ is Hermitian). Hence, $(\sqrt{M(\lambda)})'>0$.
\end{proof}

\subsection{Proof of Theorem \ref{bc2_theorem}}

We define a new vector $y(x) \in \mathbb{C}^{2n}$
so that $y (x) = (y_1 (x) \, \, y_2 (x))^t$, with $y_1 (x) = y (x)$
and $y_2 (x) = y' (x)$. In this way, we rewrite the equation in \ref{main_s}
in the form 
\begin{equation*} \label{sturm-liouville2}
\begin{aligned}
&Jy' = \mathbb{B} (x; \lambda) y; \quad
\mathbb{B} (x; \lambda) = 
\begin{pmatrix}
V(x)-\lambda f_1(x)-\lambda^2 f_2(x) & 0 \\
0 & I
\end{pmatrix}.
\end{aligned}
\end{equation*}
Let $\lambda_\infty>\max\{\frac{\sqrt{\|f_2\|_{L^{\infty} (\mathbb{R}_-)}|\nu|}}{\delta},\frac{\sqrt{\|f_{2-}\|_{L^{\infty} (\mathbb{R}_-)}|\nu_-|}}{\delta}\}$ (cf. Lemma \ref{nu}, Remark \ref{trunc}).
By Maslov Box, we mean the following sequence of contours:
(1) fix $x = 0$ and let $\lambda$
run from $0$ to $\lambda_\infty$ (the {\it top shelf});
(2) fix $\lambda = \lambda_\infty$ and let $x$ run from $0$ to $-\infty$  
(the {\it right shelf}); (3) fix $x = -\infty$ and let $\lambda$ run from $\lambda_\infty$ to $0$
(the {\it bottom shelf}); and (4) fix
$\lambda = 0$ and let $x$ run from $-\infty$ to $0$ (the {\it left shelf}). 
We denote by $\Gamma$ the simple closed
curve obtained by following each of these paths precisely once.

{\it Top shelf}. For the top shelf, we know from Lemma \ref{monotonicity1} that monotonicity
in $\lambda$ can be determined by 
$- \mathbf{X} (0; \lambda)^* J \partial_{\lambda} \mathbf{X} (0; \lambda)$, where $\mathbf{X} (0; \lambda)$ is a frame corresponding to the unstable subspace $E_-^u(0,\lambda)$, and $ \mathbf{X_\phi} (\lambda)^* J \partial_{\lambda} \mathbf{X}_\phi ( \lambda)$. We readily compute 
\begin{equation*}
\begin{aligned}
\frac{\partial}{\partial x} \mathbf{X}^* (x; \lambda) J_{2n} \partial_{\lambda} \mathbf{X} (x; \lambda)
&= (\mathbf{X}^{\prime})^* J_{2n} \partial_{\lambda} \mathbf{X} 
+ \mathbf{X}^* J_{2n} \partial_{\lambda} \mathbf{X}^{\prime} \\
&= - (\mathbf{X}^{\prime})^* J_{2n}^t \partial_{\lambda} \mathbf{X} 
+ \mathbf{X}^* \partial_{\lambda} J_{2n} \mathbf{X}^{\prime} \\
&= - \mathbf{X}^* \mathbb{B} (x; \lambda) \partial_{\lambda} \mathbf{X}
+ \mathbf{X}^* \partial_{\lambda} (\mathbb{B} (x; \lambda) \mathbf{X}) 
= \mathbf{X}^* \mathbb{B}_{\lambda} \mathbf{X}.
\end{aligned}
\end{equation*}
Integrating on $(-\infty,x]$,
we see that   
\begin{equation*}
\begin{split}
\mathbf{X} (x; \lambda)^* J_{2n} \partial_{\lambda} \mathbf{X} (x; \lambda)
&= \int_{-\infty}^x \mathbf{X} (y;\lambda)^* \mathbb{B}_{\lambda} (y; \lambda) \mathbf{X} (y; \lambda) dy\\
&=-\int_{-\infty}^x {X} (y;\lambda)^* [f_1(x)+2\lambda f_2(x)] {X} (y; \lambda) dy. 
\end{split}
\end{equation*}

Also,
\begin{equation*}
\mathbf{X_\phi} (\lambda)^* J \partial_{\lambda} \mathbf{X}_\phi ( \lambda)
= -\phi'(\lambda)>0. 
\end{equation*}
Monotonicity along the top shelf follows by setting $x = 0$ and appealing 
to condition $f_{j}>0$. 
In this way, we see that conditions $f_{j}>0$ and $\phi'(\lambda)<0$ for $\lambda\in\mathbb{R}$ ensure that 
as $\lambda$ increases the eigenvalues of $\tilde{W}_\phi (0; \lambda)$ will 
rotate in the counterclockwise direction. Therefore, $\mas (\ell (0; \cdot), \Phi(\cdot); [0,\lambda_\infty])$ is equal to the total number of intersection of the unstable subspace $E^u_-(0,\lambda)$ and the boundary subspace $\Phi(\lambda)$ for all $\lambda\geq0$, which in turn is the total geometric multiplicity of the operator pencil $\mathcal{L}_-$ (cf. \eqref{L-}) for all $\lambda\geq0$. Next, we show that all nonnegative eigenvalues of the operator pencil $\mathcal{L}_-$ are semisimple. Let a nonnegative $\lambda_0$ be an eigenvalue of $\mathcal{L}_-$ with the corresponding eigenvector $y_0\in\dom(\mathcal{L}_-(\lambda_0))$, and assume there exist a nonzero $y_1\in\dom(\mathcal{L}_-(\lambda_0))$ such that $\mathcal{L}_-(\lambda_0)y_1=-\mathcal{L}'_-(\lambda_0)y_0$. We have
\begin{equation}\label{gev}
(\mathcal{L}_-(\lambda_0)y_1,y_0)=((f_1+2\lambda_0f_2)y_0,y_0).
\end{equation}
Moreover,
\begin{equation}
(\mathcal{L}_-(\lambda_0)y_1,y_0)=y^*_0(0)y'_1(0)-(y'_0(0))^*y_1(0)+(y_1,\mathcal{L}_-(\lambda_0)y_0).
\end{equation}
Since $y_0$ and $y_1$ satisfy the boundary condition from \eqref{main_s}, and $c+\phi(\lambda_0)$ is self-adjoint, we have
\begin{equation}\label{ibp}
(\mathcal{L}_-(\lambda_0)y_1,y_0)=(y_1,\mathcal{L}_-(\lambda_0)y_0).
\end{equation}
Using \eqref{ibp} in \eqref{gev}, we arrive at
\begin{equation}\label{zero}
(y_1,\mathcal{L}_-(\lambda_0)y_0)=((f_1+2\lambda_0f_2)y_0,y_0).
\end{equation}
Since $y_0$ is the eigenvector of $\mathcal{L}_-(\lambda_0)$ corresponding to $\lambda_0$, left-hand side of \eqref{zero} is zero, but under Assumption {\bf (A1)} the right-hand side of \eqref{zero} is strictly positive, a contradiction.
Hence,
\begin{equation} \label{top_shelf}
\mas (\ell (0; \cdot), \Phi(\cdot); [0,\lambda_\infty]) = \mathcal{N} (0)+\dim(\ker(\mathcal{L}_-(0))),
\end{equation} 
where $\mathcal{N} (\lambda)$ denotes the spectral
count for (\ref{main_s}) (the number of real eigenvalues (including algebraic multiplicities) that are greater than $\lambda$).

{\it Right shelf.} Intersections between $\ell (x; \lambda)$ and $\Phi(\lambda)$ at some nonpositive value
$x=x_0$ will correspond with one or more non-trivial solutions to one of the truncated eigenvalue problems \eqref{x0} or \eqref{concoeff}. Then, according to Lemma \ref{nu} and Remark \ref{trunc}, we have

\begin{equation} \label{right_shelf}
\mas (\ell (\cdot; \lambda_\infty), \Phi(\lambda_\infty); [0,-\infty]) = 0.
\end{equation}

{\it Bottom shelf}. 
We observe that the monotonicity that we found along horizontal shelves does not immediately carry over to the bottom shelf (since that calculation is only valid for $x\in(-\infty,0]$).
We can still conclude monotonicity along the bottom shelf in the following way: by
continuity of our frames, we know that as $\lambda$ increases along the bottom shelf the eigenvalues
of $\tilde W_\phi(-\infty,\lambda)$ cannot rotate in the clockwise direction.  Moreover, eigenvalues of $\tilde W_\phi(-\infty,\lambda)$ 
cannot remain at $-1$ for any interval of $\lambda$ values (otherwise, there would exit an interval of $\lambda$ values consisting of the eigenvalues of the constant-coefficient operator pencil  $\mathcal{L}_-(\cdot)$). Therefore, 
\begin{equation*}
\mas (\ell (-\infty; \cdot), \Phi(\cdot); [\lambda_\infty,0])=-\sum_{0\leq\lambda<\lambda_\infty} \dim (E_-^u (-\infty; \lambda) \cap \Phi(\lambda)).
\end{equation*}

Next, our goal is to find all the intersections of two Lagrangian subspaces $E_-^u(-\infty,\lambda)$ and $\Phi(\lambda)$, where $E_-^u(-\infty,\lambda)$ is the unstable eigenspace of the asymptotic matrix $A_-(\lambda)$
\begin{equation*}
A_-(\lambda)=\begin{pmatrix}
0 & I \\
\lambda f_{1-}+\lambda^2 f_{2-}-V_- & 0
\end{pmatrix}.
\end{equation*}
Note that $\lambda f_{1-}+\lambda^2 f_{2-}-V_-$ is a self-adjoint holomorphic pencil, therefore, the corresponding eigenvalues denoted by $\{\nu_j(\lambda)\}_{j=1}^n$  are real for real values of $\lambda$.  We denote the corresponding eigenvectors by $\{r_j(\lambda)\}_{j=1}^n$
so that 
$(\lambda f_{1-}+\lambda^2 f_{2-}-V_- )r_j(\lambda) = \nu_j(\lambda) r_j(\lambda)$ for all 
$j \in \{1, 2, \dots, n\}$. Moreover, since $\lambda f_{1-}+\lambda^2 f_{2-}-V_-$ is a self-adjoint holomorphic pencil, the eigenvalue funcions $\{\nu_j(\lambda)\}_{j=1}^n$ can be chosen to be holomorphic for $\lambda\in\mathbb{R}$ and the corresponding eigenvectors $\{r_j\}_{j=1}^n$ can be chosen to be
orthonormal and holomorphic for $\lambda\in\mathbb{R}$ (cf. \cite[VII.2.1, p. 375]{Kato}). Also notice that  $\{\nu_j(\lambda)\}_{j=1}^n$ are positive curves for $\lambda\in\mathbb{R}_+=[0,\infty)$ (otherwise, there would exist $\mu\in\mathbb{R}$ such that $\det(-\mu^2 + V_- -\lambda f_{1-} -\lambda^2 f_{2-})=0$ which means that condition {\bf (A1)} is violated). We introduce $\mu_j^{\pm} (\lambda)$
\begin{equation*}
\begin{aligned}
\mu_j^{\pm} (\lambda) =  \mp\sqrt{\nu_j(\lambda)}
\end{aligned}
\end{equation*}
for $j = 1, 2, \dots, n$.

We note that the eigenvalues of ${A}_{-}$ are precisely the values 
$\{\mu_j^{\pm}\}_{j=1}^{n}$, and the associated eigenvectors are  
$\{{\rm r}_{\,j}^{\,\pm}\}_{j=1}^n = 
\{{r_{j} \choose {\mu_{j}^{\pm} {r_{j}}}}\}_{j=1}^n$. Therefore, two Lagrangian subspaces $E_-^u(-\infty,\lambda)$ and $\Phi(\lambda)$ intersect if and only if there exist non-zero vectors $c_1$ and $c_2$ such that $R(\lambda)c_1=c_2$ and $R(\lambda)D(\lambda)c_1=(c+\phi(\lambda))c_2$, where the columns of $R(\lambda)$ are $r_{j}(\lambda)$ and $D(\lambda)$ is diagonal with $\{\mu_j^{-}\}_{j=1}^{n}$ on the diagonal. Hence, $R(\lambda)D(\lambda)c_1=(c+\phi(\lambda))R(\lambda)c_1$. Or,
$(R(\lambda)D(\lambda)R^{-1}(\lambda)-(c+\phi(\lambda)))\tilde c_1=0$, where $\tilde c_1=R(\lambda)c_1$. Next, notice that 
\begin{equation}
(R(\lambda)D(\lambda)R^{-1}(\lambda))^2=R(\lambda)D^2(\lambda)R^{-1}(\lambda)=\lambda f_{1-}+\lambda^2 f_{2-}-V_-.
\end{equation}
Hence,
\begin{equation}
\sqrt{\lambda f_{1-}+\lambda^2 f_{2-}-V_-}=R(\lambda)D(\lambda)R^{-1}(\lambda).
\end{equation}
Consequently,  two Lagrangian subspaces $E_-^u(-\infty,\lambda)$ and $\Phi(\lambda)$ intersect if and only if the matrix pencil $M(\lambda):=\sqrt{\lambda f_{1-}+\lambda^2 f_{2-}-V_-}-(c+\phi(\lambda))$ has a zero eigenvalue. It is clear that $M$ is a continuously differentiable pencil with respect to  nonnegative parameter $\lambda$. In particular, the eigenvalue curves of $M$ are continuously differentiable pencil with respect to  nonnegative parameter $\lambda$ and when $\lambda=0$ $M(0)$ has $\mor(\sqrt{-V_-}-c)+\dim\ker(\sqrt{-V_-}-c)$ nonpositive eigenvalues. Next, notice that $\lambda f_{1-}+\lambda^2 f_{2-}-V_-$ and its derivative $ f_{1-}+2\lambda f_{2-}$ are strictly positive for $\lambda\geq0$ which in turn implies that the derivative of $\sqrt{\lambda f_{1-}+\lambda^2 f_{2-}-V_-}$ is strictly positive for $\lambda\geq0$ (cf. Lemma \ref{positivity}). Then, by Assumption {\bf (A2)}, $M'(\lambda)>0$ for $\lambda\geq0$. Hence, the eigenvalue curves $\{m_j(\lambda)\}_{j=1}^n$ of $M(\lambda)$ are strictly increasing for $\lambda\geq0$ by \cite[Theorem 5.4., p. 111]{Kato}. Moreover, by Assumption {\bf (A2)}, $-\phi(\lambda)\geq0$, consequently, $M(\lambda)\geq\sqrt{\lambda f_{1-}+\lambda^2 f_{2-}-V_-}-c$ for $\lambda\geq0$. Now, we choose $\lambda_0>0$ such that $\sqrt{\lambda_0}>\frac{\max_{\lambda\in\sigma({c)}} \lambda}{\sqrt{\min_{\lambda\in\sigma({f_{1-})}} \lambda}}$.  By the min-max principle, we know that the eigenvalues $\{\nu_j(\lambda_0)\}_{j=1}^n$ of $\lambda_0 f_{1-}+\lambda_0^2 f_{2-}-V_-$ are greater than $\lambda_0\min_{\lambda\in\sigma({f_{1-})}} \lambda$, therefore, the eigenvalues $\{\mu_j^{-}\}_{j=1}^{n}$ of  $\sqrt{\lambda f_{1-}+\lambda^2 f_{2-}-V_-}$ are greater than $\sqrt{\lambda_0\min_{\lambda\in\sigma({f_{1-})}} \lambda}$. Hence, 

\begin{equation}
M(\lambda_0)\geq\sqrt{\lambda_0 f_{1-}+\lambda_0^2 f_{2-}-V_-}-\sqrt{\lambda_0\min_{\lambda\in\sigma({f_{1-})}} \lambda}I-(c-\sqrt{\lambda_0\min_{\lambda\in\sigma({f_{1-})}} \lambda}I)>0.
\end{equation}
Hence, the eigenvalue curves of $M(\lambda)$ whose initial values at $\lambda=0$ are nonpositive eigenvalues of $M(0)$ are strictly increasing and since there exist $\lambda_0>0$ such that $M(\lambda_0)>0$, these eigenvalue curves must intersect  the $\lambda$-axis exactly once. Therefore, the number of times $M(\lambda)$ has a zero eigenvalue is equal to $\mor(\sqrt{-V_-}-c)+\dim\ker(\sqrt{-V_-}-c)$. Therefore,
\begin{equation}\label{morseM}
\mas (\ell (-\infty; \cdot), \Phi(\cdot); [\lambda_\infty,0])=-\mor(\sqrt{-V_-}-c)-\dim\ker(\sqrt{-V_-}-c).
\end{equation}

Finally, by formulas \eqref{top_shelf},  \eqref{right_shelf}, \eqref{morseM}, and fact that the Maslov index of the Maslov box is $0$, we arrive at the formula for $\mathcal{N} (0)$
\begin{align*}
\mathcal{N} (0)
&=-\mas (E_-^u(\cdot; 0), \Phi(0); [-\infty,0])-\dim(\ker(\mathcal{L}_-(0)))\\
&+\mor(\sqrt{-V_-}-c)+\dim\ker(\sqrt{-V_-}-c).
\end{align*}
Similarly, one can easily derive a formula for $\mathcal{N} (\lambda)$ for $\lambda\geq0$.
\subsection{Reality of eigenvalues}

\begin{lemma}\label{nocomplex}
	Let Assumptions {\bf (A1)} and {\bf (A3)} hold, and $\lambda=a+ib$ with $a\geq0$ be an eigenvalue of the operator pencil \eqref{L-}. Then $b=0$, that is, $\lambda\in\mathbb{R}$.
\end{lemma}
\begin{proof}
	After multiplying \eqref{main_s} by the corresponding eigenvector $y$ and integration by parts, we arrive at
	\begin{align}\label{real}
	\begin{split}
	&\big((c+\phi(\lambda))y(0),y(0)\big)  -(y',y')+ (Vy,y) = \lambda (f_1y,y)+\lambda^2 (f_2y,y).
	\end{split}
	\end{align}
	Next, we take the imaginary part of \eqref{real}
	\begin{align}
	\begin{split}
	&(\Im\phi(\lambda)y(0),y(0))= b (f_1y,y)+2ab (f_2y,y).
	\end{split}
	\end{align}
	It follows from Assumption {\bf (A1)} that $(f_1y,y)+2a (f_2y,y)>0$, therefore, the sign of the right hand side is $\sign b=\sign\Im\lambda$. 
	By Assumption {\bf (A3)}, the matrix $\Im\phi(\lambda)$ is semidefinite, with sign opposite to $\sign \Im \lambda= \sign b$. 
	Therefore, the sign of the left hand side is also of (indefinite) sign opposite to $\sign\Im \lambda$.
	Comparing signs of lefthand and righthand sides, we find that $b=0$.
\end{proof}

\hfill $\square$

\br\label{phirmk}
Assumption {\bf (A1)} on $\Im \phi$ is sharp in Lemma \ref{nocomplex}, as without it
one may readily construct counterexamples for operator pencils independent of $\lambda$.
For polynomial $\phi(\lambda) = \sum_{j=1}^r c_j \lambda^j$, 
{\bf (A1)} on $\Im \phi$ implies that $r=1$, or linearity, 
as may be seen by looking at the large $|\lambda|$ limit, for which the
highest term $c_r\lambda^r$ dominates $\sgn \Im \phi(\lambda)$.
\er

\section{Proof of Theorem \ref{bc3_theorem} } \label{proofs}

\subsection{Upper Bound on the Spectrum of \eqref{mainL}}

By Lemma \ref{essspecL}, we know that  the real part of the essential spectrum of \eqref{mainL} is bounded above by $\gamma<0$. Next, we show that a set of the real isolated eigenvalues of \eqref{mainL} is bounded above.

\begin{lemma}\label{Vinf}
	Assume {\bf (A4)} . Then there exists $\nu\in\BbbR$ such that for all real eigenvalues $\lambda$ of \eqref{mainL}
	\begin{equation}
	\lambda\le\frac{\sqrt{\|f_2\|_{L^{\infty} (\mathbb{R})}\|V\|_{L^{\infty} (\mathbb{R})}}}{\delta}.
	\end{equation}
\end{lemma}

\begin{proof}
	Let $\lambda$ be a real eigenvalue of \eqref{mainL} with the corresponding eigenvector $y$. Then 
	\begin{align} 
	\begin{split}
	&y'' + V(x) y = \lambda f_1(x) y+\lambda^2 f_2(x) y; 
	\quad x\in\BbbR_L,\\
	&y(L)=0.
	\end{split}
	\end{align}
	Or, after multiplying by $y$ and integration by parts, we arrive at
	\begin{align} 
	\begin{split}
	&-(y',y')+ (Vy,y) = \lambda (f_1y,y)+\lambda^2 (f_2y,y).
	\end{split}
	\end{align}
	Or,
	\begin{align} 
	\begin{split}
	(f_2y,y) \lambda^2 +(f_1y,y) \lambda+(y',y')-(Vy,y)=0.
	\end{split}
	\end{align}
	Therefore, $\lambda$ satisfies one of the following equalities
	\begin{align}\label{discL}
	\begin{split}
	\lambda=\frac{-(f_1y,y)\pm\sqrt{(f_1y,y)^2-4(f_2y,y)[(y',y')-(Vy,y)]}}{2(f_2y,y)}
	\end{split}
	\end{align}
	
	If $\lambda$ satisfies the equality with the negative sign in front of the square root, then $\lambda$ is nonpositive.
	Next, we assume that $\lambda$  satisfies the equality with the positive sign in front of the square root.
	Next, we estimate the following quadratic form $\mathcal{M}$ with the domain $H^1(\mathbb{R}_L)$:
	\begin{equation}
	\mathcal{M}[y]=(y',y')-(Vy,y)\ge- \|V\|_{L^{\infty} (\mathbb{R})} \|y\|_{L^2 (\mathbb{R}_L)}^2.
	\end{equation}

	Therefore,
	\begin{align*}
	\begin{split}
	&-(f_1y,y)+\sqrt{(f_1y,y)^2-4(f_2y,y)[(y',y')-(Vy,y)]}\\
	&\le-(f_1y,y)+\sqrt{(f_1y,y)^2+4(f_2y,y)\|V\|_{L^{\infty} (\mathbb{R})}\|y\|_{L^2 (\mathbb{R}_L)}^2]}\le2\sqrt{\|f_2\|_{L^{\infty} (\mathbb{R})}|\|V\|_{L^{\infty} (\mathbb{R})}|}\|y\|_{L^2 (\mathbb{R}_L)}.
	\end{split}
	\end{align*}
	Therefore, $\lambda\le\frac{\sqrt{\|f_2\|_{L^{\infty} (\mathbb{R})}\|V\|_{L^{\infty} (\mathbb{R})}}}{\delta}$.
\end{proof}

\begin{remark}\label{unifL}
	Note that the upper bound from Lemma \ref{Vinf} is independent of $L$.
\end{remark}

In this section, we use our Maslov index framework to prove our
main theorems.

We define a new vector $y(x) \in \mathbb{C}^{2n}$
so that $y (x) = (y_1 (x) \, \, y_2 (x))^t$, with $y_1 (x) = y (x)$
and $y_2 (x) = y' (x)$. In this way, we rewrite \ref{mainL}
in the form 
\begin{equation*} \label{sturm-liouville2}
\begin{aligned}
&Jy' = \mathbb{B} (x; \lambda) y; \quad
\mathbb{B} (x; \lambda) = 
\begin{pmatrix}
V(x)-\lambda f_1(x)-\lambda^2 f_2(x) & 0 \\
0 & I
\end{pmatrix}.
\end{aligned}
\end{equation*}

\subsection{Proof of Theorem \ref{bc3_theorem}}
Let $\lambda_\infty>\frac{\sqrt{\|f_2\|_{L^{\infty} (\mathbb{R})}\|V\|_{L^{\infty} (\mathbb{R})}}}{\delta}$ (cf. Lemma \ref{Vinf}). By Maslov Box, we mean the following sequence of contours:
(1) fix $x = -\infty$ and let $\lambda$ run from $\lambda_\infty$ to $0$
(the {\it bottom shelf}); 
(2) fix
$\lambda = 0$ and let $x$ run from $-\infty$ to $L$ (the {\it left shelf});  (3) fix $x = L$ and let $\lambda$
run from $0$ to $\lambda_\infty$ (the {\it top shelf});  and (4) fix $\lambda = \lambda_\infty$ and let $x$ run from $L$ to $-\infty$  
(the {\it right shelf}). 
We denote by $\Gamma$ the simple closed
curve obtained by following each of these paths precisely once.


{\it Bottom shelf}. We begin our analysis with the bottom shelf. Since 
$E_-^u(-\infty,\lambda)$ does not intersect the Dirichlet subspace $\mathcal{D}$, we see that in fact the matrix 
$\det{(\tilde{W}_D (0; \lambda)+I)}$ does not vanish, and so 
\begin{equation} \label{bottom_shelfR}
\mas (\ell (-\infty; \cdot), \mathcal{D}; [\lambda_\infty, 0]) = 0.
\end{equation}

{\it Left shelf}. It is clear that $\mathbf{X}_D^* J \partial_{x} \mathbf{X}_D=0$, but
$- \mathbf{X} (x; 0)^* J \partial_{x} \mathbf{X} (x; 0)$ is not sign definite for values of $x$ which means that can not directly apply Lemma \ref{monotonicity1}. Instead, we can compute the spectral flow of  $\tilde{W}_D (\cdot; 0)$ through $-1$. Assume that at least one of the eigenvalues of $\tilde{W}_D (\cdot; 0)$ at $x=x_*$ is $-1$ ($E^u_-(x_*,0)$ and $\mathcal{D}$ has a non-trivial intersection). Then the spectal flow of $\tilde{W}_D (\cdot; 0)$ through $-1$ as $x$ crosses through $x_*$ is determined by signature of the following quadratic form defined on $\ker(\tilde{W}_D (x_*; 0)+I_n)$ (cf. \cite{HS18}):
\begin{align}
\begin{split}
\tilde{\mathcal{Q}}(w) &= 
- 2 \Big(\Big( (X - i Y)^{-1}  \Big)^* \mathbf{X} (x_*; 0)^* J \partial_{x} \mathbf{X} (x_*; 0) (X - i Y)^{-1}  w, 
w \Big)_{\mathbb{C}^n} \\
&=  2 \Big(\Big( (X - i Y)^{-1}  \Big)^*(X (x_*)^* Y' (x_*) - Y (x)^* X'(x_*)) (X - i Y)^{-1}  w, 
w \Big)_{\mathbb{C}^n}\\
&=2 \Big(\Big( (X - i Y)^{-1}  \Big)^*(X (x_*)^* (-V)X(x_*) - Y (x_*)^* Y(x_*)) (X - i Y)^{-1}  w, 
w \Big)_{\mathbb{C}^n}.
\end{split}
\end{align}
Since $(X - i Y)^{-1}  w\in\ker(X(x_*))$, we have the following formula for $\tilde{\mathcal{Q}}$:
\begin{align}
\begin{split}
\tilde{\mathcal{Q}}(w) &= 
- 2 \Big(\Big( (X - i Y)^{-1}  \Big)^* \mathbf{X} (x_*; 0)^* J \partial_{x} \mathbf{X} (x_*; 0) (X - i Y)^{-1}  w, 
w \Big)_{\mathbb{C}^n} \\
&=  2 \Big(\Big( (X - i Y)^{-1}  \Big)^*(X (x_*)^* Y' (x_*) - Y (x)^* X'(x_*)) (X - i Y)^{-1}  w, 
w \Big)_{\mathbb{C}^n}\\
&=-2 \Big( Y (x_*) (X - i Y)^{-1}  w, 
 Y (x_*)(X - i Y)^{-1} w \Big)_{\mathbb{C}^n}<0.
\end{split}
\end{align}

Therefore,
\begin{equation}\label{leftR}
\begin{aligned}
\mas (\ell (\cdot; 0), \mathcal{D}; [-\infty, L])
& =- \sum_{-\infty < x \le L} \dim (\ell (x; 0) \cap \mathcal{D}) \\
&=- \sum_{-\infty < x \le L} \dim \ker (\mathbf{X} (x; 0)^* J \mathbf{X}_D)\\
&=-\sum_{-\infty < x <L} \dim \ker (\mathbf{X} (x; 0)^* J \mathbf{X}_D)-\dim(\ker(\mathcal{L}_L(0))).
\end{aligned}
\end{equation}

{\it Top shelf}. Since $\mathbf{X}_D^* J \partial_{\lambda} \mathbf{X}_D=0$, by Lemma \ref{monotonicity1}, monotonicity
in $\lambda$ can be determined by 
$- \mathbf{X} (L; \lambda)^* J \partial_{\lambda} \mathbf{X} (L; \lambda)$, 
and we readily compute 
\begin{equation*}
\begin{aligned}
\frac{\partial}{\partial x} \mathbf{X}^* (x; \lambda) J_{2n} \partial_{\lambda} \mathbf{X} (x; \lambda)
&= (\mathbf{X}^{\prime})^* J_{2n} \partial_{\lambda} \mathbf{X} 
+ \mathbf{X}^* J_{2n} \partial_{\lambda} \mathbf{X}^{\prime} \\
&= - (\mathbf{X}^{\prime})^* J_{2n}^t \partial_{\lambda} \mathbf{X} 
+ \mathbf{X}^* \partial_{\lambda} J_{2n} \mathbf{X}^{\prime} \\
&= - \mathbf{X}^* \mathbb{B} (x; \lambda) \partial_{\lambda} \mathbf{X}
+ \mathbf{X}^* \partial_{\lambda} (\mathbb{B} (x; \lambda) \mathbf{X}) 
= \mathbf{X}^* \mathbb{B}_{\lambda} \mathbf{X}.
\end{aligned}
\end{equation*}
Integrating on $(-\infty,x]$,
we see that   
\begin{equation*}
\begin{split}
\mathbf{X} (x; \lambda)^* J_{2n} \partial_{\lambda} \mathbf{X} (x; \lambda)
&= \int_{-\infty}^x \mathbf{X} (y;\lambda)^* \mathbb{B}_{\lambda} (y; \lambda) \mathbf{X} (y; \lambda) dy\\
&=-\int_{-\infty}^x {X} (y;\lambda)^* [f_1(x)+2\lambda f_2(x)] {X} (y; \lambda) dy. 
\end{split}
\end{equation*}

Monotonicity along the top shelf follows by setting $x = L$ and appealing 
to condition $f_{j}>0$. 
In this way, we see that condition $f_{j}>0$ ensures that 
as $\lambda$ increases the eigenvalues of $\tilde{W}_D (L; \lambda)$ will 
rotate in the counterclockwise direction. Therefore, $\mas (\ell (L; \cdot), \mathcal{D}; [0, \lambda_\infty])$ is equal to the total number of intersection of the unstable subspace $E^u_-(L,\lambda)$ and the boundary subspace $\mathcal{D}$ for all $\lambda\geq0$, which in turn is the total geometric multiplicity of the operator pencil $\mathcal{L}_L$ (cf. \eqref{LL}) for all $\lambda\geq0$. Next, we show that all nonnegative eigenvalues of the operator pencil $\mathcal{L}_L$ are semisimple. Let a nonnegative $\lambda_0$ be an eigenvalue of $\mathcal{L}_L$ with the corresponding eigenvector $y_0\in\dom(\mathcal{L}_L(\lambda_0))$, and assume there exist a nonzero $y_1\in\dom(\mathcal{L}_L(\lambda_0))$ such that $\mathcal{L}_L(\lambda_0)y_1=-\mathcal{L}'_L(\lambda_0)y_0$. We have
\begin{equation*}
(\mathcal{L}_L(\lambda_0)y_1,y_0)=((f_1+2\lambda_0f_2)y_0,y_0).
\end{equation*}
After ingratiating by parts, we arrive at
\begin{equation}\label{zeroL}
(y_1,\mathcal{L}_L(\lambda_0)y_0)=((f_1+2\lambda_0f_2)y_0,y_0).
\end{equation}
Since $y_0$ is the eigenvector of $\mathcal{L}_L(\lambda_0)$ corresponding to $\lambda_0$, left-hand side of \eqref{zeroL} is zero, but under Assumption {\bf (A4)} the right-hand side of \eqref{zeroL} is strictly positive, a contradiction.
Hence,
\begin{equation}\label{topR}
\mas (\ell (L; \cdot), \mathcal{D}; [0, \lambda_\infty])=\mathcal{N} (0)+\dim(\ker(\mathcal{L}_L(0))),
\end{equation}
where $\mathcal{N} (0)$ denotes the spectral
count for (\ref{mainL}) (the number of real eigenvalues (including algebraic multiplicities) that are greater than $0$).

{\it Right shelf.}  Intersections between $\ell (x; \lambda)$ and $\mathcal{D}$ at some value
$x=x_0$, where $-\infty<x_0\leq L$ will correspond with one or more non-trivial solutions to the truncated eigenvalue problem:
\begin{align}\label{unL}
\begin{split}
&y'' + V(x) y = \lambda f_1(x) y+\lambda^2 f_2(x) y; 
\quad x\in\BbbR_{x_0}:=(-\infty, x_0],\,\,x_0\leq L,\\
&y(x_0)=0.
\end{split}
\end{align}
By Lemma \ref{Vinf} and Remark \ref{unifL}, we have the uniform upper bound $\frac{\sqrt{\|f_2\|_{L^{\infty} (\mathbb{R})}\|V\|_{L^{\infty} (\mathbb{R})}}}{\delta}$ for the real eigenvalues of \eqref{unL}, and $\ell (-\infty; \lambda_\infty)$ and $\mathcal{D}$ don not intersect by the bottom shelf argument. Therefore,
\begin{equation} \label{right_shelfR}
\mas (\ell (\cdot; \lambda_\infty), \mathcal{D}); [L, -\infty]) = 0.
\end{equation}


Finally, by formulas \eqref{bottom_shelfR},  \eqref{leftR}, \eqref{topR}, \eqref{right_shelfR},   and fact that the Maslov index of the Maslov box is $0$, we arrive at the formula for $\mathcal{N} (0)$
\begin{align*}
\mathcal{N} (0)
&=\sum_{-\infty < x <L} \dim (E_-^u(x,0) \cap \mathcal{D}).
\end{align*}
Similarly, one can easily derive a formula for $\mathcal{N} (\lambda)$ for $\lambda\geq0$.

\subsection{Proof of Theorem \ref{bc1_theorem}}

\begin{proof}
	
	We follow the proof of similar results from \cite{BCJLMS,SaS00}. Our goal is to compute the number of positive eigenvalues of the operator pencil $\mathcal{L}(\cdot)$, that is, the number of $\lambda\in(0,\lambda_\infty)$ such that $$\mathbb E_+^s(0,\lambda)\wedge\mathbb E_-^u(0,\lambda)=0.$$ On the other hand,  $\mathcal{N} (0)$ for the operator pencil $\mathcal{L}_L$ is equal to the number of zeros of the function
	\[
	\mathcal D\wedge \mathbb E_-^u(L,\lambda),\,\,\,\lambda\in(0,\lambda_\infty).
	\]
	We claim that $\mathbb E_+^s(0,\lambda)\wedge\mathbb E_-^u(0,\lambda)$ and $\mathcal D\wedge \mathbb E_-^u(L,\lambda)$ have the same number of zeros, counting multiplicity, for sufficiently large values of $L$.
	
	Let $\phi(x_1,x_2;\lambda)$ denote the propagator of the non-autonomous differential equation $y' = A(x,\lambda)y$. Also denote $\mathcal D_L(\lambda) = \phi(0,L;\lambda)\mathcal D\wedge\mathbb E_-^u(0,\lambda)$ and $\mathcal D_\infty(\lambda) = \mathbb E_+^s(0,\lambda)\wedge\mathbb E_-^u(0,\lambda)$, and choose an analytic basis $\{v_j^+(\lambda)\}$ of $\mathbb E_+^s(0,\lambda)$. Note that $\mathcal D\oplus\mathbb E_+^u(+\infty,\lambda) = \mathbb{C}^{2n}$ because $\mathbb E_+^u(+\infty,\lambda)\cap\mathcal D = \{0\}$. 
	
	It is known that $\mathbb E_+^{s/u}(L,\lambda)\to\mathbb E_+^{s/u}(+\infty,\lambda)$ exponentially as $L\to+\infty$; see \cite[Thm. 1]{SaS00}. Then, as in \cite[Thm. 2]{SaS00}, there exist unique vectors $w_j^+(\lambda)\in\mathbb E_+^u(L,\lambda)$ such that $\mathcal D=\text{span}\{\phi(L,0;\lambda)v_j^+(\lambda)+w_j(\lambda):j=1,\ldots,n\}$ and
	\[
	\phi(0,L;\lambda) \mathcal D = \text{span}\{v_j^+(\lambda)+\phi(0,L;\lambda)w_j^+(\lambda): j=1,\ldots,n\}.
	\]
	Thus $\phi(0,L;\lambda)\mathcal D$ and $\mathbb E_+^s(0,\lambda)$ are $\exp(-\sigma_+L)$-close, where $\sigma_+$ is the rate of exponential  decay of solutions at $+\infty$. Then $\mathcal D_L(\lambda)$ and $\mathcal D_\infty(\lambda)$ have the same multiplicities of zeros by \cite[Rmk 4.3]{SaS00}. The claim now follows from the fact that $\mathbb E_-^u(L,\lambda) = \phi(L,0;\lambda)\mathbb E_-^u(0,\lambda)$, hence
	\[
	\mathcal D_{L}= \phi(0,L;\lambda)\mathcal D\wedge\phi(0,L;\lambda)\mathbb E_-^u(L,\lambda) = [\det\phi(0,L;\lambda)]\mathcal D\wedge \mathbb E_-^u(L,\lambda).
	\]
	
	In particular, $\mathcal{N} (0)$ for the operator pencil $\mathcal{L}_L$ is independent of $L$ for $L$ large enough. Finally, applying Theorem \ref{bc3_theorem}, we infer the main assertion.

\end{proof}

\section{Application}

We study spectral stability of hydraulic shock profiles 
of the (inviscid) Saint-Venant equations for inclined shallow-water flow:
\ba \label{sv}
\d_th+\d_xq&=0,\\
\d_tq+\d_{x}\left(\frac{q^2}{h}+\frac{h^2}{2F^2}\right)&=h-\frac{|q|q}{h^2},
\ea
where $h$ denotes fluid height; $q=hu$ total flow, with $u$ fluid velocity; and
$F>0$ the {\it Froude number}, a nondimensional parameter depending on reference height/velocity and inclination.

Following \cite{YZ}, we here focus on the hydrodynamically stable case $0<F<2$, and associated {\it hydraulic shock profile} solutions 
\be\label{prof}
(h,q)(x,t)= (H,Q)(x-ct), \quad \lim_{z\to - \infty}(H,Q)(z)= (H_L,Q_L), \; \lim_{z\to - \infty}(H,Q)(z)= (H_R,Q_R).
\ee
These are piecwise smooth traveling-wave solutions satisfying the Rankine-Hugoniot jump and Lax entropy conditions at any discontinuities. Their existence theory reduces to the study of an explicitly solvable 
scalar ODE with polynomial coefficients \cite{YZ}

We now turn to the discussion of stability.
Linearizing \eqref{sv} about a smooth profile $(H,Q)$ following \cite{MZ,SYZ18}, we obtain eigenvalue equations
\be 
\label{syseval}
Av'=(E-\lambda \Id-A_x)v,
\ee
where
\ba
A&=\left[\begin{array}{cc} -c & 1\\ \frac{H}{F^2}-\frac{Q^2}{H^2} & \frac{2Q}{H}-c \end{array}\right],\quad E=\left[\begin{array}{cc} 0 & 0\\ \frac{2Q^2}{H^3}+1 & -\frac{2Q}{H^2} \end{array}\right].\\
\ea
It is shown in \cite{YZ} that essential spectrum of $\mathcal{L}:= -A\partial_x- \partial_xA +E$ is confined to
$\{\lambda:  \Re \lambda <0\}\cup \{0\}$, with an embedded eigenvalue at $\lambda=0$.
Moreover, it is shown that the embedded eigenvalue at $\lambda=0$ is of multiplicity one in a 
generalized sense defined in terms of an associated Evans function defined as in \cite{AGJ,MZ}.
It follows by the general theory of \cite{MZ2} relating generalized, or Evans-type, spectral stability to linearized and nonlinear stability, that smooth hydraulic shock profiles are nonlinearly orbitally stable so long as they are
{\it weakly spectrally stable} in the sense that there exist no decaying solutions of \eqref{syseval}
on $\{\lambda:  \Re \lambda \geq 0\}\setminus \{0\}$.

The discontinuous case is more complicated, involving a free boundary with transmission/evolution conditions given by the Rankine-Hugoniot jump conditions.
However, following the approach of Erpenbeck-Majda for the study of such problems in the context
of shocks and detonations, one may deduce a generalized eigenproblem consisting of the same ODE \eqref{syseval},
but posed on the negative half-line $x\in (-\infty,0)$ with boundary condition 
\be \label{bc}
[\lambda\overline{W}-R(\overline{W})]_\perp \cdot A(0^-)v(0^-)=0,
\ee
where $\overline{W}:=(H,Q)^T$ and $[h]:= h(0^+)-h(0^-)$ denotes jump in $h$ across $x=0$; see \cite{YZ} for further details.
Similarly as in the smooth case, it is shown in \cite{YZ} that
essential spectrum of $\mathcal{L}$ with boundary condition \eqref{bc}
is confined to $\{\lambda:  \Re \lambda <0\}\cup \{0\}$, with an embedded eigenvalue at $\lambda=0$, of multiplicity
one in a generalized sense defined by an associated Evans-Lopatinsky function.
It follows by the general theory of \cite{YZ} that discontinuous hydraulic shock profiles are 
nonlinearly orbitally stable so long as they are weakly spectrally stable in the sense that there exist 
no decaying solutions of \eqref{syseval}-\eqref{bc} on $\{\lambda:  \Re \lambda \geq 0\}\setminus \{0\}$.

In summary, by the analytical results of \cite{MZ2,YZ}, the question of nonlinear stability of hydraulic shock profiles 
has been reduced in both smooth and discontinuous case 
to determination of weak spectral stability, or nonexistence of eigenvalues $\lambda\neq 0$
with $\Re \lambda \geq 0$ of eigenvalue problem \eqref{syseval} on the whole- or half-line, respectively.

The special structure exploited here is that the eigenvalue system \eqref{syseval} may be reduced to a scalar second-order system of generalized
Sturm-Liouville type.
Specifically, following the general approach described in \cite{SYZ18}, the eigenvalue system \eqref{syseval} originating from any $2\times 2$
relaxation system may converted to a scalar second-order equation
\be\label{geneval}
y''+V(x)y= \lambda f_1(x)y + \lambda^2 f_2(x) y.
\ee

In the half-line case, there is in addition a $\lambda$-dependent Robin-type boundary condition
\be\label{robin}
y'(0)=(c_1 + c_2\lambda)y(0), 
\ee
where $f_1(x), f_2(x)>0$, $c_1, c_2<0$, $V(x)<\delta<0$, and Assumptions {\bf (A1)-(A4)} are satisfied for the half- and whole line, respectively. Moreover, $E^u_-(\cdot,0)$ does not intersect $\mathcal{D}$ for the whole line case, and $E^u_-(\cdot,0)$ does not intersect $\colspan{1 \choose c_1}$ for the half line case \cite{SYZ18}. Thus, applying Theorems \ref{bc2_theorem} and \ref{bc1_theorem}, we obtain the following result:

\bt 
Nondegenerate hydraulic shock profiles of the Saint-Venant equations \eqref{sv} are weakly spectrally stable,
across the entire range of existence.
\et

\section{Discussion and open problems}\label{s:disc}
Eigenvalue problems of form \eqref{main}, \eqref{main_s}
were studied in \cite{SYZ18} in connection with stability of hydraulic shock profiles,
or asymptotically constant traveling-wave solutions $w(x,t)=W(x-ct)$ of
the inclined Saint-Venant equations, a $2\times2$ first-order hyperbolic relaxation system
of form 
\be\label{relax}
w_t + f(w)_x= R(w), \quad R=(r,0)^T.
\ee
The eigenvalue equations associated with $W$ are of form $(Aw)'= (E - \lambda) w  $,
where $A(x):=(df/dw-c\Id)(W(x))$ and $E(x):=(dR/dw)(W(x))$.
Solving for one coordinate of $w$ as a linear function of $\lambda$ and the other coordinate
yields a second-order scalar problem in the second coordinate, now quadratic in its dependence on $\lambda$;
see \cite[(1.8), (1.9)]{SYZ18}.
More generally, eigenvalue problems with possibly nonlinear dependence on $\lambda$ are standard in Evans function 
literature \cite{AGJ}, which treats generalized eigenvalue problems of the first-order form $w'=A(\lambda x)w$,
with $A$ analytic in $\lambda$ but not necessarily linear.
For solution by rather different techniques in the fourth-order scalar case
of a quadratic eigenvalue problem related to stability of phase-transitional shock waves, see \cite{Z00}.

In \cite{SYZ18}, the associated eigenvalue problems were shown to be stable, by a combination of classical Sturm--Liouville
techniques, and by-hand arguments making use of special structure as needed.
Here, we generalize and systematize this approach using Maslov index techniques,
to obtain a full Sturm--Liouville theorem giving an exact eigenvalue count in the general case.
The methods used in \cite{HS,HS18} to obtain spectral counts 
of operators on a bounded interval as particularly close to the point of view followed here.
At the same time, we extend the theory from scalar to vector with Hermitian coefficient case,
a task involving interesting issues (Lemma \ref{positivity})
related to monotone matrix functions and L\"owner's theorem \cite{L34};
for further discussion, see Appendix \ref{s:loewner}.

In the scalar case, our results answer the problem posed in \cite{SYZ18} of determining minimal structural
requirements under which one can obtain a complete Sturm--Liouville theorem counting unstable eigenvalues.
In the system case, an interesting open problem is to extend our results to the general, non-Hermitian coefficient case.
We note that even in the Hermitian-coefficient system case, it is not clear how to determine analytically
the number of conjugate points; however, numerical counting gives an attractive alternative to numerical 
Evans function computations/winding number calculations, as described, e.g., in \cite{Z11}.
A second very interesting open problem, noted in \cite{SYZ18} is to determine whether the assumptions of our
theory developed apply to shock profiles of general $2\times 2$ relaxation systems of the type considered in \cite{L87},
and if so, whether these are always stable (as in the Saint Venant case \cite{SYZ18}) or whether one can
find examples of spectrally unstable smooth or discontinuous profiles for amplitudes sufficiently large.

	\appendix

\section{Monotone matrix functions and L\"owner's theorem}\label{s:loewner}
In this appendix, we explore relations between Lemma \ref{positivity} 
and the theory of monotone matrix operators and L\"owner's theorem \cite{L34}.

\subsection{Monotonicity of $f(A)=A^p$, $0<p\leq 1$}\label{s:mon}
We first prove (a variant of) the standard result of monotonicity of $A\to A^p$ (proof adapted from \cite{U}),
in the process establishing a strict convex interpolation inequality for families of commuting matrices.

\bl[Monotonicity of the geometric mean]\label{halflem}
Let $A<B$, $C\leq D$. $A,B,C,D$ symmetric positive definite, and let $A,C$ and $B,D$ commute.  Then,
$$
(AC)^{1/2}<(BD)^{1/2}.
$$
\el

\begin{proof}
$A<B$ and $C<D$ implies $|B^{-1/2}AB^{-1/2}|<1$ and $D^{-1/2}CD^{-12}|\leq 1$, 
	which in turn gives
	$
	|B^{-1/2}AB^{-1/2} D^{-1/2}CD^{-1/2}|<1,
	$
	and thus
	$
	\rho( B^{-1/2}AB^{-1/2} D^{-1/2}CD^{-1/2} ) < 1,
	$
	where $\rho(\cdot)$ denotes spectral radius
	and $|\cdot|$ denotes matrix norm.

	By similiarity, this implies
	$\rho(C^{1/2} D^{-1/2} B^{-1/2}AB^{-1/2} D^{-1/2}C^{1/2} ) < 1,$
	hence, by commutativity of $A,C$ and $B,D$,
	$$
	\rho\big(C^{1/2} (BD)^{-1/2} A^{1/2} (C^{1/2} (BD)^{-1/2} A^{1/2})^*\big)<1,
	$$
or $| C^{1/2} (BD)^{-1/2} A^{1/2}|<1$.  By similarity, this is equivalent to 
$\rho \big((AC)^{1/2} (BD)^{-1/2}\big) <1$, or $ (AC)^{1/2}< (BD)^{1/2} $ as claimed.
\end{proof}

\bc[Matrix interpolation]\label{1cor}
Let $A<B$, $C\leq D$. $A,B,C,D$ symmetric positive definite, and let $A,C$ and $B,D$ commute.  Then,
$A^pC^{1-p}< B^pD^{1-p}$ for all $0< p\leq 1$.
\ec

\begin{proof}
By repeated application of Lemma \ref{halflem}, we obtain the result for any dyadic $p$, giving 
$A^pC^{1-p}\leq  B^pD^{1-p}$ for general $p$ by continuity.
Noting that any $0< p\leq 1$ may be expressed as the geometric mean of a dyadic $0<p_1\leq p$ and a general 
$p\leq p_2\leq 1$, we obtain strict inequality for general $p$ as well.
\end{proof}

\bc[Monotonicity of $A^p$ \cite{L34}]\label{moncor}
For $A<B$, $A^p<B^p$ for any $0<p\leq 1$.
\ec

\begin{proof}
	Take $C=D=\Id$ in Corollary \ref{1cor}.
\end{proof}

\br
From Corollary \ref{moncor}, we obtain already \emph{nonnegativity}, $(A^p)'\geq 0$ for $A'\geq 0$,
of the derivative of the matrix function $f(A)=A^p$, for any $0\leq p\leq 1$.
\er

\subsection{Connection to L\"owner's matrix}\label{s:lconn}
%
%
%
%

\bpr[\cite{L34}]\label{fprop}
Let $A(t)>0$ be symmetric and $R(t)$ an orthogonal matrix of eigenvectors of $A(t)$,
with $A(t)R(t)=R(t)D(t)$, $D$ diagonal, and $f$ differentiable. Then 
\be\label{form}
\big( R^T(d/dt)(f(A) R\big)_{jk}= 
\big(R^T A'R\big)_{jk} \frac{f(d_j)-f(d_k)}{d_j-d_k}. 
\ee
\epr

\begin{proof}
	From
	\be\label{eig}
	A(t)R(t)= R(t)D(t), \quad R^T(t)A(t)=D(t) R^T(t), 
	\ee
	we obtain, differentiating \eqref{eig}(i), $A'R+ AR'-R'D-RD'=0$, whence, applying $R^T$ on the left and using
	\eqref{eig}(ii), we get
	$$
	\begin{aligned}
	0=R^TA'R+ R^TAR'- R^T R'D- R^TRD'&= R^TA'R+ D R^TR'- R^T R'D- R^TRD'\\
				  &=
	R^TA'R+ D R^TR'- R^T R'D- D'.
	\end{aligned}
	$$
	From this we may deduce that 
\be\label{facts}
D'=\diag R^TA'R;\quad
( R^TR')_{jk}= \frac{R^TA'R }{d_j-d_k}, 
	\, j\neq k.
		\ee
	
		Differentiating $f(A)=Rf(D)R^T$, gives 
		$(A^{1/2})'=R'D^{1/2}R^T + Rf'(D)'R^T + Rf(D) (R^T)'$, whence,
multiplying on the left by $R^T$ and the right by $R$, and using $R^TR=\Id$ and $(R^T)'R=-R^TR'$, we have
\ba
R^T f(A)'R&=R^TR'f(D) + f'(D) + f(D) (R^T)'R  \\
			       &= f'(D) + \Big(R^TR' f(D) - f(D)R^T R' \Big). 
\ea
Combining this with \eqref{facts} then gives \eqref{form}.
\end{proof}

\begin{definition}
	The L\"owner matrix is defined as $L_{jk}= \frac{f(d_j)-f(d_k)}{d_j-d_k}$.
\end{definition}

\bc[\cite{L34}] The matrix function $f(A)$ is nonstrictly monotone, $(d/dt)f(A)\geq 0$ for $(d/dt)A\geq 0$,
if and only if the L\"owner matrix $L_{ji}$ is positive semidefinite.
\ec

\begin{proof}
Since $P:=R^TA'R\geq 0$ if and only if $A'\geq 0$, this is equivalent to the statement that
$Q_{jk}:=L_{jk}P_{jk}\geq 0$ for all symmetric $P\geq 0$.
Assume that $Q\geq 0$ for any $P\geq 0$. Then 
in particular, we have for any vector $v$, taking $P=vv^T$, that
$$
x^TQX=
\sum_{j,k} y_j L_{jk} y_k\geq 0,
$$
$y_j:= v_j x_j$, for all choices of $x$, $v$, hence all choices of $y_j$.
This gives $L_{jk}\geq 0$.
On the other hand, if $L_{jk}\geq 0$, then, expanding any symmetric $P\geq 0$ as
$P=\sum_i \mu_i v^i (v^i)^T$, $\mu_i\geq 0$ we have, setting $y^i_j:=v^i_j x_j$,
\be\label{ref}
x^TQx= \sum_i \mu_i  \sum_{jk} y^i_j L_{jk} y^i_k\geq 0.
\ee
\end{proof}

\bpr[Positivity of $f(A)'$]\label{newprop} The matrix function $f(A)$ satisfies $(d/dt)f(A)> 0$ for $(d/dt)A> 0$,
if and only if the L\"owner matrix $L_{ji}$ is positive semidefinite and $f'(t)>0$.
\epr

\begin{proof}
	By \eqref{ref}, and $L_{k}\geq 0$, we have $Q>0$ if and only if $L y^i\not \equiv 0$ for all $i$ for
$y^i_j:= v^i_j x_j$  and all choices of $x$, $v^i$, where $L$ is the L\"owner matrix associated with $D$.
By considering $A$ diagonal, we find that $f'>0$ is a necessary condition, along with semidefiniteness of
$L$ as established in Corollary \ref{moncor}.
To see that they are sufficient, note that $f'>0$ implies that the coefficients of $L$ are positive.
By the Frobenius--Perron theorem, therefore, it has a principal eigenvector $w$ with positive entries $w_j$,
and $w$ has eigenvalue $\nu>0$.
Thus,
$$
x^TQx= \sum_i \mu_i \sum_{jk}  y^i_j L_{jk} y^i_k\geq 
\sum_i \mu_i \nu \frac{(\sum_{j}  y^i_j w_j )^2}{|w|}>0
$$
unless $0= \sum_j y^i_j w_j= \sum_j v^i_j (w_j x_j) $ for all $i$.  As $\{v^i\}$ is a basis, this
would imply $w_j x_j=0$, which, by $w_j>0$, would imply $x_j=0$ for all $j$, or $x=0$.
Thus, $Q>0$ and we are done.
\end{proof}

\bc\label{newcor} The matrix function $f(A)=A^p$ has positive derivative, $(d/dt)f(A)> 0$ for $A>0$, $(d/dt)A> 0$,
for all $0<p\leq 1$.
\ec

\begin{proof}
	By Corollary \ref{moncor}, $f$ is nonstrictly monotone, hence $f(A)'\geq 0$ and $L\geq 0$. Since
	$f'>0$ by inspection, we are done.
\end{proof}

\br
The conclusions and methods regarding nonstrict monotonicity are standard.
However, our conclusions regarding strict positivity of $f(A)'$ so far as we know are new.
\er

\subsection{Implications}\label{s:imp}
The conclusions of Corollaries \ref{moncor}, \ref{newcor} imply interesting inequalities on the
associated L\"owner matrices. For example, in the case of the square root function $f(A)=A^{1/2}$,
the associated L\"owner matrix is
$L_{jk}= 
\frac{ d_j^{1/2}-d_k^{1/2}}{d_j-d_k}= \frac{ 1}{d_j^{1/2}+d_k^{1/2}}$,
which must therefore be semidefinite.
We conjecture that for every dimension $n$, and $d_j$ distinct, 
$$
\det \Big( \frac{1}{d^{1/2}_j+d^{1/2}_k}\Big)=  \frac{\Pi_{j> k} (d^{1/2}_j-d^{1/2}_k)^2} {\Pi_{j > k} (d^{1/2}_j+d^{1/2}_k)^2\Pi_j 
2d^{1/2}_j},
$$
giving positive definiteness of $L$ by induction on principal minors.

\section{Essential spectrum}\label{essspect}
First, we consider the limiting operator pencil $\mathcal{L_{-\infty}}(\lambda)$ and the corresponding first order operator pencil
$\mathcal{T_{-\infty}}(\lambda)$:
\begin{align}\label{Linf}
\begin{split}
&\mathcal{L_{-\infty}}(\lambda):\dom(\mathcal{L}_{-\infty}(\lambda))\subset( L^2(\BbbR_-))^n\to( L^2(\BbbR_-))^n,\\
&\mathcal{L_{-\infty}}(\lambda)y:=y'' + V_ -y- \lambda f_{1-} y-\lambda^2 f_{2-} y, \,y\in\dom(\mathcal{L}_{-\infty}(\lambda)),
\quad x\in\BbbR_-,\\
&\dom(\mathcal{L}_{-\infty}(\lambda))=\{y\in( H^2(\BbbR_-))^n:(c+\phi(\lambda))y(0)- y'(0)=0\},\\
&\mathcal{T_{-\infty}}(\lambda):\dom(\mathcal{T}_{-\infty}(\lambda))\subset( L^2(\BbbR_-))^{2n}\to( L^2(\BbbR_-))^{2n},\\
&\mathcal{T_{-\infty}}(\lambda)Y:=Y'-A_{-}(\lambda)Y, \,\,A_{-}(\lambda) = 
\begin{pmatrix}
0 & I \\
\lambda f_{1-}+\lambda^2 f_{2-}-V_- & 0
\end{pmatrix}, \,Y\in\dom(\mathcal{T}_{-\infty}(\lambda)),\\
&\dom(\mathcal{T}_{-\infty}(\lambda))=\{y\in( H^1(\BbbR_-))^{2n}:(c+\phi(\lambda)\,\,\, -I)Y(0)=0\}.\\
\end{split}
\end{align}
When $A_{-}(\lambda)$ is hyperbolic, its stable  $E^s_-(\lambda)$ and unstable $E^u_-(\lambda)$ subspaces yield direct sum decomposition of $\BbbC^{2n}$. We denote by $P^s_-(\lambda)$ and $P^u_-(\lambda)$ the corresponding eigenprojections. Moreover, in this case, the system $Y'=A_{-}(\lambda)Y$ possesses the exponential dichotomy on $\BbbR_-$.

Let $\{\nu_j(\lambda)\}_{j=1}^n$ denote the eigenvalues of the matrix pencil $\lambda f_{1-}+\lambda^2 f_{2-}-V_-$. We introduce $\{\mu_j^{\pm} (\lambda)\}_{j=1}^n$
\begin{equation*}
\begin{aligned}
\mu_j^{\pm} (\lambda) =  \mp\sqrt{\nu_j(\lambda)}
\end{aligned}
\end{equation*}
that are are precisely the eigenvalues of ${A}_{-}(\lambda)$.  Hence,  $A_{-}$ is not hyperbolic at $\lambda\in\BbbC$ if and only if $\det(-\mu^2 + V_- -\lambda f_{1-} -\lambda^2 f_{2-})=0$ for some $\mu\in\BbbR$.  In particular, Assumption {\bf (A1)} guaranties that there exists an open subset denoted by $\Omega$ containing the closed right half plane that consists of the points $\lambda$ such that  $A_{-}$ is hyperbolic and $n^u_-(\lambda)=\dim E^u_-(\lambda)=n$.

Next, we look for an $H^1(\BbbR_-)$ solution of $Y'=A_{-}(\lambda)Y+F$, where $F\in( L^2(\BbbR_-))^{2n}$. In what follows, we will suppress $\lambda$ dependence. By variation of parameters formula, we have
\begin{align*}
Y(x)=e^{A_-x}Y_0+\int_0^xe^{A_-(x-t)}F(t)dt,\,\,\,x\leq0,
\end{align*}
where $Y_0$ is the initial data. Or,
\begin{align*}
Y(x)=e^{A_-x}P_-^uY_0+e^{A_-x}P_-^sY_0+\int_0^xe^{A_-(x-t)}P_-^uF(t)dt+\int_0^xe^{A_-(x-t)}P_-^sF(t)dt.
\end{align*}
Finally, we can rewrite it as follows:
\begin{align*}
Y(x)=e^{A_-x}P_-^uY^-_0+e^{A_-x}P_-^sY^-_0-\int^0_xe^{A_-(x-t)}P_-^uF(t)dt+\int_{-\infty}^xe^{A_-(x-t)}P_-^sF(t)dt,
\end{align*}
where 
\begin{equation}
Y^-_0=Y_0-\int_{-\infty}^0e^{-A_-t}P_-^sF(t)dt.
\end{equation}
Once again, we can rewrite the solution $Y$ by using the Green's function \\ $G(z)=\begin{cases}
-e^{A_-z}P_-^u, & \text{if $z\leq0$},\\
e^{A_-z}P_-^s, & \text{$z>0$}.
\end{cases}$
\begin{align*}
Y(x)=e^{A_-x}P_-^uY^-_0+e^{A_-x}P_-^sY^-_0+(G*\hat F)(x),
\end{align*}
where $\hat F(x)=\begin{cases}
F(x), & \text{if $x\leq0$},\\
0, & \text{$x>0$}.
\end{cases}$. Note that $\|G*\hat F\|_2\leq C\|G\|_2\|F\|_2$. Then the solution $Y$ belongs to $(L^2(\BbbR_-))^{2n}$ if and only if $Y^-_0:=Y(0)-\int_{-\infty}^0e^{-A_{-}t}P_-^sF(t)dt\in E^u_-$, that is,
\begin{align}\label{sol}
Y(x)=e^{A_-x}P_-^uY^-_0+(G*\hat F)(x).
\end{align}

Fix $\lambda\in\Omega$ and denote by $E(\lambda)$ the $2n\times2n$ matrix $((I \,\,c+\phi(\lambda))^T\,\,\,v_1^u(\lambda)\,\,\,\ldots\,\,\,v_{n}^u(\lambda))$, where $n=\dim E^u_-$ and $\{v_j^u\}_{j=1}^{n}$ form a basis for $E^u_-$. If $\det(E(\lambda))\neq0$, then there exists $\alpha\in\BbbC^{2n}$ such that 
\begin{equation}\label{B4}
E(\lambda)\alpha=\int_{-\infty}^0e^{-tA_-}P_-^sF(t)dt,
\end{equation}
which guaranties the existence of the solution $Y$ of $Y'=A_{-}(\lambda)Y+F$ that satisfies the boundary condition at $0$ and such that $Y^-_0\in E^u_-$, therefore, by formula \eqref{sol},
$Y\in H^1(\BbbR_-)$ and $\lambda$ belongs to the resolvent set of the operator pencil  $\mathcal{T_{-\infty}}(\cdot)$. Similarly,  let $F=(0,f)^T$, where $f\in( L^2(\BbbR_-))^{n}$ and $\det(E(\lambda))\neq0$, then the existence of $y\in\dom(\mathcal{L}_{-\infty}(\lambda))$ such that $\mathcal{L}_{-\infty}(\lambda)y=f$, therefore, 
$\lambda$ belongs to the resolvent set of the operator pencil  $\mathcal{L_{-\infty}}(\cdot)$. 

Before we prove the next lemma, we introduce the adjoint operator pencils $\mathcal{T}^*_{-\infty}(\lambda)$ $\mathcal{L}^*_{-\infty}(\lambda)$:
\begin{align}\label{Linfadj}
	\begin{split}
		&\mathcal{L^*_{-\infty}}(\lambda):\dom(\mathcal{L}^*_{-\infty}(\lambda))\subset( L^2(\BbbR_-))^n\to( L^2(\BbbR_-))^n,\\
		&\mathcal{L^*_{-\infty}}(\lambda)y:=y'' + V_ -y- \bar\lambda f_{1-} y-\bar\lambda^2 f_{2-} y, \,y\in\dom(\mathcal{L}^*_{-\infty}(\lambda)),
		\quad x\in\BbbR_-,\\
		&\dom(\mathcal{L}_{-\infty}(\lambda))=\{y\in( H^2(\BbbR_-))^n:(c+\phi^*(\lambda))y(0)- y'(0)=0\},\\
		&\mathcal{T^*_{-\infty}}(\lambda):\dom(\mathcal{T}_{-\infty}(\lambda))\subset( L^2(\BbbR_-))^{2n}\to( L^2(\BbbR_-))^{2n},\\
		&\mathcal{T^*_{-\infty}}(\lambda)Y:=Y'+A^*_{-}(\lambda)Y, \,\,A^*_{-}(\lambda) = 
		\begin{pmatrix}
			0 & \bar\lambda f_{1-}+\bar\lambda^2 f_{2-}-V_- \\
			I & 0
		\end{pmatrix}, \,Y\in\dom(\mathcal{T}^*_{-\infty}(\lambda)),\\
		&\dom(\mathcal{T}^*_{-\infty}(\lambda))=\{y\in( H^1(\BbbR_-))^{2n}:(I\,\,\, c+\phi^*(\lambda))Y(0)=0\}.\\
	\end{split}
\end{align}
Furthermore,  $F\in\ran(\mathcal{T^*_{-\infty}}(\lambda))$ if and only if there exists $\alpha\in\BbbC^{2n}$ such that 
\begin{equation}\label{adj}
\hat E(\lambda)\alpha=\int_{-\infty}^0e^{-tA^*_-}(I-(P^{s}_-)^*)F(t)dt,
\end{equation}
where $\hat E(\lambda)$ denotes the $2n\times2n$ matrix $(( c+\phi^*(\lambda)\,\,-I)^T\,\,\,\hat v_1^u(\lambda)\,\,\,\ldots\,\,\,\hat v_{n}^u(\lambda))$, where $n=\dim (\ran(I-(P^{u}_-)^*))$ and $\{\hat v_j^u\}_{j=1}^{n}$ form a basis for $\ran(I-(P^{u}_-)^*)$, where $I-(P^{u}_-)^*$ is the exponential dichotomy projection  for the system $Y'=-A^*_-(\lambda)Y$ on $\BbbR_-$. 

The following lemma holds:
\begin{lemma}\label{Fredlim}
	Let Assumption {\bf (A1)} hold and fix $\lambda\in\Omega$. Then $\ran(\mathcal{L_{-\infty}}(\lambda))$ and $\ran(\mathcal{T_{-\infty}}(\lambda))$ are closed and 
	\begin{align*}
	\dim(\ker(\mathcal{L_{-\infty}}(\lambda)))&=\dim(\ker(\mathcal{T_{-\infty}}(\lambda)))=	\dim(\ker(E(\lambda))),\,\,\,\\\codim(\ran(\mathcal{L_{-\infty}}(\lambda)))&=\codim(\ran(\mathcal{T_{-\infty}}(\lambda)))=	\codim(\ran(E(\lambda))).
	\end{align*}
	Moreover, $\mathcal{L_{-\infty}}(\lambda)$ and $\mathcal{T_{-\infty}}(\lambda)$ are Fredholm with index $0$.
\end{lemma}
\begin{proof}
	It is clear from \eqref{sol} that $F\in\ran(\mathcal{T_{-\infty}}(\lambda))$ if and only if $\int_{-\infty}^0e^{-tA_-}P_-^sF(t)dt\in\ran(E(\lambda))$. Since $\ran(E(\lambda))$ is closed and $F\to\int_{-\infty}^0e^{-tA_-}P_-^sF(t)dt$ is continuous in $( L^2(\BbbR_-))^{2n}$, it follows that $\ran(\mathcal{T_{-\infty}}(\lambda))$ is closed. Similarly, by choosing $F=(0,f)^T$ and constructing a continuous map $f\to\int_{-\infty}^0e^{-tA_-}P_-^sF(t)dt$ in $( L^2(\BbbR_-))^{n}$, we deduce that $\ran(\mathcal{L_{-\infty}}(\lambda))$ is closed. 
	
	Also, it is clear that $y\in\ker(\mathcal{L_{-\infty}}(\lambda))$ if and only if $(y,y')^T\in\ker(\mathcal{T_{-\infty}}(\lambda))$ and both are in one-to-one correspondence with an $\alpha\in\ker(E(\lambda))$ (note that $\ker(E(\lambda))=\colspan{I \choose c+ \phi(\lambda)}\cap\ran(P_-^u(\lambda))$).
	
	Finally, we know that $\codim(\ran E(\lambda))=\dim(\ker(E^*(\lambda)))$ and
	\begin{align}\label{kernel}
	\begin{split}
	\ker(E^*(\lambda))&=(\colspan{I \choose c+ \phi(\lambda)}\cup\ran(P_-^u(\lambda)))^\perp\\
	&=\colspan^\perp{I \choose c+ \phi(\lambda)}\cap\ran^\perp(P_-^u(\lambda))\\
	&=\colspan{c+ \phi^*(\lambda) \choose -I}\cap\ran(I-(P^{u}_-(\lambda))^*)=\ker\hat E(\lambda).
	\end{split}
	\end{align}
	Since it is clear from \eqref{adj} that $y\in\ker(\mathcal{L^*_{-\infty}}(\lambda))$ if and only if $(-y',y)^T\in\ker(\mathcal{T^*_{-\infty}}(\lambda))$ and both are in one-to-one correspondence with an $\alpha\in\ker(\hat E(\lambda))$, by \eqref{kernel}, we have 
	\begin{equation}
	\dim(\ker(\mathcal{L^*_{-\infty}}(\lambda)))=	\dim(\ker(\mathcal{T^*_{-\infty}}(\lambda)))=\dim(\ker \hat E(\lambda))=\dim(\ker(E^*(\lambda))),
	\end{equation}

	 Finally, $\mathcal{L_{-\infty}}(\lambda)$ and $\mathcal{T_{-\infty}}(\lambda)$ are Fredholm with index $0$ due to the following identity:
	 \begin{equation*}
	 \dim(\ker(E(\lambda)))-	\codim(\ran(E(\lambda)))=(2n-\codim(\ker(E(\lambda))))-(2n-\dim(\ran(E(\lambda))))=0.
	 \end{equation*}
\end{proof}

Now, we would like to mimic the above analysis for the operator pencil $\mathcal{L_-}(\cdot)$. Assumption {\bf (A1)} guaranties the existence of exponential dichotomy on $\BbbR_-$ for $\lambda\in\Omega$ for the system:
\begin{align}\label{sys}
Y'=A(x,\lambda)Y, \,\,A(x,\lambda) = 
\begin{pmatrix}
0 & I \\
\lambda f_{1}(x)+\lambda^2 f_{2}(x)-V(x) & 0
\end{pmatrix},
\end{align} 
which is due to the roughness theorem of
exponential dichotomies. That is, there exist a projection $P$ and constants $K_i>0, \alpha_i>0$ such that for all $x,t\in\BbbR_-$
\begin{align}
\begin{split}
|U(x)PU^{-1}(t)|&\leq K_1e^{-\alpha_1(x-t)},\,\,\,t\leq x,\\
|U(x)(1-P)U^{-1}(t)|&\leq K_2e^{-\alpha_2(t-x)},\,\,\,t\geq x,\\
\end{split}
\end{align}
where $U(x)$ ($U(0)=I$) is the fundamental matrix for \eqref{sys}.
\begin{align*}
Y(x)=U(x)Y_0+\int_0^xU(x)U^{-1}(t)F(t)dt,\,\,\,x\leq0,
\end{align*}
where $Y_0$ is the initial data. Or,
\begin{align*}
Y(x)=U(x)PY^-_0+U(x)(1-P)Y^-_0-\int^0_xU(x)(1-P)U^{-1}(t)F(t)dt+\int_{-\infty}^xU(x)PU^{-1}F(t)dt,
\end{align*}
where 
\begin{equation}
Y^-_0=Y_0-\int_{-\infty}^0PU^{-1}(t)F(t)dt.
\end{equation}

Once again, we can rewrite the solution $Y$ by using the Green's function \\ $G(x,t)=\begin{cases}
-U(x)(1-P)U^{-1}(t), & \text{if $x\leq t$},\\
U(x)PU^{-1}(t), & \text{$x>t$}.
\end{cases}$
\begin{align*}
Y(x)=U(x)PY^-_0+U(x)(1-P)Y^-_0+\int_{-\infty}^0G(x,t)F(t)dt,
\end{align*}
where the integral term on the right hand side is $L^2$-integrable with respect to $x$. 
Then the solution $Y$ belongs to $(L^2(\BbbR_-))^{2n}$ if and only if $Y^-_0:=Y_0-\int_{-\infty}^0PU^{-1}(t)F(t)dt\in E^u_-:=\ran(I-P)$, that is,
\begin{align}\label{sol}
Y(x)=U(x)(1-P)Y^-_0+\int_{-\infty}^0G(x,t)F(t)dt.
\end{align}

Fix $\lambda\in\Omega$ and denote by $E_-(\lambda)$ the $2n\times2n$ matrix $((I \,\,c+\phi(\lambda))^T\,\,\,v_1^u(\lambda)\,\,\,\ldots\,\,\,v_{n}^u(\lambda))$, where $n=\dim E^u_-=\dim\ran(I-P)$ and $\{v_j^u\}_{j=1}^{n}$ form a basis for $E^u_-=\ran(I-P)$. If $\det(E_-(\lambda))\neq0$, then there exists $\alpha\in\BbbC^{2n}$ such that 
\begin{equation}
E_-(\lambda)\alpha=\int_{-\infty}^0PU^{-1}(t)F(t)dt,
\end{equation}
which guaranties the existence of $Y(0)$ that satisfies the boundary condition at $0$ and such that $Y^-_0\in E^u_-=\ran(I-P)$, therefore, 
$\lambda$ belongs to the resolvent set of the operator pencil $\mathcal{T_{-}}(\cdot)$. Similarly,  let $F=(0,f)^T$, where $f\in( L^2(\BbbR_-))^{n}$ and $\det(E(\lambda))\neq0$, then the existence of $y\in\dom(\mathcal{L}_{-\infty}(\lambda))$ such that $\mathcal{L}_{-\infty}(\lambda)y=f$, therefore, 
$\lambda$ belongs to the resolvent set of the operator pencil  $\mathcal{L_{-\infty}}(\cdot)$. 

Furthermore, the following lemma holds:
\begin{lemma}\label{b3}
	Let Assumption {\bf (A1)} hold and fix $\lambda\in\Omega$.Then $\ran(\mathcal{L_{-}}(\lambda))$ and $\ran(\mathcal{T_{-}}(\lambda))$ are closed and 
	\begin{align*}
	\dim(\ker(\mathcal{L_{-}}(\lambda)))&=\dim(\ker(\mathcal{T_{-}}(\lambda)))=	\dim(\ker(E_-(\lambda))),\,\,\,\\\codim(\ran(\mathcal{L_{-}}(\lambda)))&=\codim(\ran(\mathcal{T_{-}}(\lambda)))=	\codim(\ran(E_-(\lambda))).
	\end{align*}
	Moreover, $\mathcal{L_{-}}(\lambda)$ and $\mathcal{T_{-}}(\lambda)$ are Fredholm with index $0$.
\end{lemma}
\begin{proof}
	The proof is similar to that of Lemma \ref{Fredlim}, and a key relation is
	\begin{equation*}
	F\in\ran(\mathcal{T}_-(\lambda)) \iff  \int_{-\infty}^0PU^{-1}(t)F(t)dt\in\ran(E_-(\lambda)).
	\end{equation*}
\end{proof}

Let us recall the definition of multiplicity of eigenvalues of nonlinear pencils (cf. \cite{BLR,MM}).
\begin{definition}
	Let $\lambda_0$ be an eigenvalue of the pencil $\mathcal{T}(\cdot)$. 
	\begin{enumerate}
		\item A tuple $(v_0,\ldots,v_{n-1})\in(\dom(\mathcal{T}(\lambda_0)))^n$ is called a chain of generalized eigenvectors (CGE) of $\mathcal{T}(\cdot)$  at $\lambda_0$ if the polynomial $v(\lambda)=\sum_{j=0}^{n-1}(\lambda-\lambda_0)^jv_j$ satisfies
		\begin{equation*}
		(\mathcal{T}v)^{(j)}(\lambda_0)=0,\,\,j=\overline{1,n-1}.
		\end{equation*}
		The order of the chain is the index $r_0$ satisfying
		\begin{equation*}
		(\mathcal{T}v)^{(j)}(\lambda_0)=0,\,\,j=\overline{1,r_0-1},\,\,\,(\mathcal{T}v)^{(r_0)}(\lambda_0)\neq0.
		\end{equation*}
		The rank $r(v_0)$ of a vector $v_0\in\ker(\mathcal{T}(\lambda_0)), v_0\neq0$, is the maximum order of CGEs starting at $v_0$.
		\item A canonical system of generalized eigenvectors (CSGE) of $\mathcal{T}(\cdot)$ at $\lambda_0$ is a system of vectors 
		\begin{equation*}
		v_{j,p}\in\dom(\mathcal{T}(\lambda_0)), \,\,j=\overline{0,\mu_p-1},\,p=\overline{1,q},
		\end{equation*}
		with the following properties:
		\begin{enumerate}
			\item $v_{0,1},\ldots,v_{0,q}$ form a basis of $\ker(\mathcal{T}(\lambda_0))$,
			\item the tuple $(v_{0,p},\ldots,v_{\mu_p-1,p})$ is a CGE of of $\mathcal{T}(\cdot)$ at $\lambda_0$ for $p=\overline{1,q}$,
			\item for $p=\overline{1,q}$ the indices $\mu_p$ satisfy
			\begin{equation*}
			\mu_p=\max\{r(v_0): v_0\in\ker(\mathcal{T}(\lambda_0))\setminus\hbox{span}\{v_{0,\nu}:1\leq\nu<p\}.
			\end{equation*}
			\item The number $\mu_1+\ldots+\mu_q$ is called the algebraic multiplicity of $\lambda_0$.
		\end{enumerate}
	\end{enumerate}
\end{definition}

\begin{lemma}\label{b3}
	Let Assumption {\bf (A1)} hold. Then $\Omega\subset\BbbC\setminus\sigma_{ess}(\mathcal{L_{-}})$. Moreover, $\Omega$ consists of either points of the resolvent set  or isolated eigenvalues of finite algebraic multiplicity of the operator pencil $\mathcal{L_{-}}(\cdot)$.
\end{lemma}

\begin{proof}
	Fix $\lambda\in\Omega$. Then, by Lemma \ref{b3}, $\lambda\notin\sigma_{ess}(\mathcal{L_{-}})$. Therefore, $\lambda$ is either a point of the resolvent set of $\mathcal{L_{-}}(\cdot)$ or an eigenvalue of $\mathcal{L_{-}}(\cdot)$. Moreover, $\lambda$ is an eigenvalue of $\mathcal{L_{-}}(\cdot)$ if and only if $\lambda$ is an eigenvalue of $\mathcal{T_{-}}(\cdot)$ if and only if it is a root of the analytic function $\det(E_-(\lambda))$. Therefore, all the eigenvalues from $\Omega$ are isolated. Moreover, one can show that $\mathcal{L}^{-1}_{-}(\cdot)$ is meromorphic in $\Omega$ and the order of the pole at the eigenvalue $\lambda_0$ is the algebraic multiplicity of $\lambda_0$ (cf. \cite{BLR,MM}). In particular, one can use the functional analytic approach of combining the differential operator and
the boundary operator to a two-component operator defined on a fixed space, not
depending on the eigenvalue parameter, that is,
\begin{align*}
\begin{split}
&\hat{\mathcal{L}}_-(\lambda)\in\mathcal{B}(H^2(\BbbR_-),( L^2(\BbbR_-))^n\times\BbbC^n),\\
&\hat{\mathcal{L}}_-(\lambda)y:=\begin{pmatrix}
y'' + V(x) y- \lambda f_1(x) y-\lambda^2 f_2(x) y \\
(c+\phi(\lambda))y(0)- y'(0)
\end{pmatrix}.
\end{split}
\end{align*}
\end{proof}

\begin{lemma}\label{b4}
	Let Assumption {\bf (A4)} hold. Then $\Omega\subset\BbbC\setminus\sigma_{ess}(\mathcal{L_{}})$. Moreover, $\Omega$ consists of either points of the resolvent set  or isolated eigenvalues of finite algebraic multiplicity of the operator pencil $\mathcal{L_{}}(\cdot)$.
\end{lemma}

\begin{proof}
	One can prove the result similar to Lemma \ref{b3} for the full line problem. In this case, one would use $E^s_+(\lambda)$ instead of $\colspan{I \choose c+ \phi(\lambda)}$, and a key relation is
	\begin{equation*}
	F\in\ran(\mathcal{T}(\lambda)) \iff  \int_{-\infty}^0PU^{-1}(t)F(t)dt+ \int^{\infty}_0(1-Q)U^{-1}(t)F(t)dt\in\ran(E(\lambda)),
	\end{equation*}
	where $\mathcal{T}(\lambda)$ is the first-order operator pencil associated with the eigenvalue problem \eqref{main}, $P$ and $Q$ are the dichotomy projections on $\BbbR_-$ and $\BbbR_+$, respectively, and $E(\lambda)=\ran(I-P(\lambda))\wedge\ran(Q(\lambda))$.
Then the proof is similar to that of Lemma \ref{b3}.
\end{proof}

\end{document}